\documentclass{amsart}
\usepackage{latexsym,amssymb,amsmath,textcomp}
\usepackage[english]{babel}
\usepackage{url}
\usepackage{amsfonts}

\def\Z{\mathbb{Z}}

\def\N{\mathbb{N}}

\def\D{\mathcal{D}}

\def\Kmn{K_{m\times n}}
\def\B{\mathcal{B}}
\def\X{\mathcal{X}}
\def\Y{\mathcal{Y}}
\def\R{\mathcal{R}}
\def\ZZ{\mathcal{Z}}
\def\Fa{\mathcal{F}}
\def\G{\Gamma}

\def\s{\square}
\def\d{\bigstar}

\def\ccup{\mathop{\cup}}
\newcommand{\ii}[1]{\left\lVert{#1}\right\rVert}

\newtheorem{defini}{Definition}[section]
\newtheorem{prop}[defini]{Proposition}
\newtheorem{lem}[defini]{Lemma}

\newtheorem{cor}[defini]{Corollary}
\newtheorem{thm}[defini]{Theorem}

\theoremstyle{definition}
\newtheorem{rem}[defini]{Remark}
\newtheorem{ex}[defini]{Example}

\begin{document}

\title[Cyclic uniform $2$-factorizations of the complete multipartite graph]{Cyclic uniform $2$-factorizations\\ of the complete multipartite graph}

\author[A. Pasotti]{Anita Pasotti}
\address{DICATAM - Sez. Matematica, Universit\`a degli Studi di
Brescia, Via
Branze 43, I-25123 Brescia, Italy}
\email{anita.pasotti@unibs.it}

\author[M.A. Pellegrini]{Marco Antonio Pellegrini}
\address{Dipartimento di Matematica e Fisica, Universit\`a Cattolica del Sacro Cuore, Via
Musei 41,
I-25121 Brescia, Italy}
\email{marcoantonio.pellegrini@unicatt.it}

\begin{abstract}
The generalization of the Oberwolfach Problem, proposed by J. Liu in 2000, asks for a uniform $2$-factorization of 
 the complete multipartite graph $\Kmn$.
Here we focus our attention on $2$-factorizations regular under the cyclic group $\Z_{mn}$,
whose $2$-factors are disjoint union of cycles all of even length $\ell$.
In particular, we present a complete solution for the extremal cases $\ell=4$ and $\ell=mn$. 
\end{abstract}


\maketitle

\section{Introduction}
Throughout this paper, $\Kmn$  will denote
the complete multipartite graph with $m$ parts of same cardinality $n$.
If $n=1$,
we may identify $K_{m\times 1}$ with the complete graph on $m$ vertices $K_m$,
while $K_{1\times n}$ is a union of $n$ disjoint vertices.
So from now we consider $K_{m\times n}$ with $m,n>1$.
Also, note that
$K_{m\times 2}$ is nothing but the cocktail party graph $K_{2m}-I$, namely the graph obtained from $K_{2m}$ by removing a
$1$-factor $I$, that is, a set of $m$ pairwise disjoint edges.

For any graph $\Gamma$
we write $V(\G)$ for the set of its vertices and $E(\G)$ for the set of its edges.
We denote by $(c_0,c_1,\ldots,c_{\ell-1})$ the cycle of length $\ell$ whose edges are
$[c_0,c_1],[c_1,c_2],\ldots,[c_{\ell-1},c_0]$. An $\ell$-\emph{cycle system} of a graph $\G$ is a set $\B$ of cycles of length $\ell$ whose
edges partition $E(\G)$; clearly a graph may admit a cycle system only if the degree of each its vertex is even.

A $2$-\emph{factor} of a graph $\G$ is a set of cycles whose vertices partition $V(\G)$.
A $2$-\emph{factorization} of $\G$ is a set $\Fa$ of $2$-factors such that any edge of $\G$
appears in exactly one member of $\Fa$. Hence the cycles appearing in the $2$-factors of $\Fa$
form, altogether, a cycle system of $\G$, called the \emph{underlying cycle system} of $\Fa$.
A $2$-factorization whose $2$-factors are all isomorphic to a given $2$-factor is said to be \emph{uniform}.
When the underlying cycle system of $\Fa$ consists of $\ell$-cycles, the uniform $2$-factorization is
called a ${C_{\ell}}$-factorization.
In this paper we focus our attention on the case $\ell$ even. In particular
we deal with the extremal cases: $\ell=4$ and $\ell=mn$.
In the second case one speaks of a hamiltonian $2$-factorization, 
since each $2$-factor is a hamiltonian cycle of $\Kmn$.

The problem of finding $2$-factorizations of $\Kmn$ is a natural generalization of the
analogous problem for the complete graph, i.e. of the  well-known Oberwolfach Problem (OP) proposed by
G. Ringel in 1967.
Even if OP has been devoted of an extensive research activity, the complete solution has not yet been
achieved. However, many partial existence results are known, see \cite{BDD} and the references therein.
On the other hand very few results have been obtained about $2$-factorizations of $\Kmn$.
The generalization of OP to the multipartite case was proposed by J. Liu in \cite{L2000}:
\begin{quote}
``At a gathering
there are $m$ delegations each having $n$ people. Is it possible to arrange a seating of $mn$ people
present at $s$ round tables $T_1, T_2,\ldots, T_s$ (where each $T_i$ can accommodate $\ell_i\geq3$ people and
$\sum \ell_i= mn$) for a suitable number of different meals so that each person has every other person not in the same
delegation for a neighbor exactly once?''
\end{quote}
Liu gave a  complete solution for the existence of $C_\ell$-factorizations of $\Kmn$ in \cite{L2000,L2003}.
Bryant et al. in \cite{BDP} completely solved the case  of uniform $2$-factorizations of $\Kmn$ whose
$2$-factors are disjoint union of cycles of even length.
We point out that other generalizations of OP have been considered, for example, see \cite{CEZKVE,EZTVE,G,LL,OP,P}.
Here we consider $C_\ell$-factorizations of $\Kmn$ regular under the action of the cyclic group $\Z_{mn}$.

In general, given an additive group $G$ of order $v$ and a graph $\G$ such that $V(\G)=G$
one can consider the regular action of $G$ on $V(\G)$ defined by $x \mapsto x+g$, for any
$x\in V(\G)$ and any $g \in G$.
A $2$-factorization $\Fa$ of $\G$ is said to be \emph{regular under the action of $G$}
if  for any $F\in\Fa$ and any $g\in G$, we have also
$F+g\in\Fa$.
If $G$ is a cyclic group a $2$-factorization regular under $G$ is simply said \emph{cyclic}.
A natural reason to looking for a regular $2$-factorization $\Fa$ is that
 this additional property allows to economize, sometimes considerably, the description of $\Fa$; it
suffices in fact to give a complete system of representatives for the $G$-orbits
on $\Fa$ rather than to list all the $2$-factors of $\Fa$.

We point out that OP has been solved in several cases constructing regular $2$-factorizations,
for instance see \cite{BDF2004,BR,M}. The $2$-factorizations obtained in \cite{BDP,L2000,L2003}
concerning the multipartite case do not have this property. In this paper we provide the first
infinite classes of regular (in particular cyclic) $C_\ell$-factorizations of $\Kmn$.
The main results of the paper are the following theorems, that completely solve the cases $\ell=4$ and $\ell=mn$ even.

\begin{thm}\label{thm:C4}
A cyclic $C_4$-factorization of $\Kmn$ exists if and only if one of the following cases occurs:
\begin{itemize}
\item[(\rm{a})] $m,n$ are both even;
\item[(\rm{b})] $m$ is odd,  $n\equiv 0\pmod 4$ and  $(m,n)\neq (p^\alpha,4)$, where $p$ is a prime.
\end{itemize}
\end{thm}

\begin{thm}\label{thm:Cmn}
Let $mn$ be even.
A cyclic hamiltonian $2$-factorization of $\Kmn$ exists if and only if
\emph{all} these three conditions are satisfied:
\begin{itemize}
\item[(\rm{a})] $n$ is even;
\item[(\rm{b})] if $n\equiv2\pmod4$, then  $m\equiv 1,2\pmod 4$;
\item[(\rm{c})] if $n=2$, then $m\neq p^\alpha$ where $p$ is an odd prime.
\end{itemize}
\end{thm}

The paper is organized as follows. In Section \ref{sec2}, we give some preliminaries results
and we talk about the concept of a $2$-starter which plays a fundamental role in all our constructions.
Then, in Section \ref{sec3}, we provide necessary conditions for the existence of cyclic $C_{\ell}$-factorizations
and in Section \ref{sec4} we introduce some auxiliary sets and notation which allow us to illustrate the constructions
in a simpler and more elegant way. Sections \ref{sec5} and \ref{sec6} contain direct constructions of cyclic $C_4$-factorizations of $\Kmn$. In Section \ref{sec7} there are direct constructions of cyclic hamiltonian $2$-factorizations of $\Kmn$.
Finally, in Section \ref{sec8} we prove Theorems \ref{thm:C4} and \ref{thm:Cmn}, using the results of  the previous sections.

\section{Preliminaries}\label{sec2}

As remarked in the Introduction, a necessary condition for the existence of a $2$-factorization of 
$K_m$ is $m-1$ even. Hence, a $C_4$-factorization of the complete graph cannot exist.\\
On the other hand, about cyclic hamiltonian $2$-factorizations the following results are known.

\begin{thm}\label{th:BurDel}\cite{BDF2004}
There exists a cyclic hamiltonian $2$-factorization of $K_m$ if and only if $m$ is odd but $m\neq15, p^\alpha$
with $p$ prime and $\alpha>1$.
\end{thm}

\begin{thm}\cite{JM}\label{JM}
There exists a cyclic hamiltonian $2$-factorization of $K_{m\times 2}$
if and only if $m\equiv 1,2\pmod 4$ and $m\neq p^\alpha$ with $p$ odd prime and
$\alpha\geq 1$.
\end{thm}

\begin{thm}\label{thm:MPP}\cite{MPP}
Let $m$ be even; a cyclic  hamiltonian $2$-factorization of $\Kmn$ exists if and only if
\begin{itemize}
\item[(\rm{a})] $n$ is even, and
\item[(\rm{b})] if $n\equiv 2 \pmod 4$, then $m\equiv 2 \pmod 4$.
\end{itemize}
\end{thm}

Now, we recall the definition of a \emph{Cayley graph on a finite group $(G,+)$ with connection set $\Omega$}, denoted by
$Cay[G:\Omega]$.
Let  $\Omega \subseteq G\setminus \{0\}$
such that for every $\omega \in \Omega$ we also have $-\omega \in \Omega$.
The Cayley graph $Cay[G:\Omega]$ is the graph whose vertices are the elements of $G$ and in which two vertices are adjacent if
and only if their difference is an element of $\Omega$.
Note that $\Kmn$ can be interpreted as the Cayley graph $Cay[\Z_{mn}:\Z_{mn}\setminus m\Z_{mn}]$,
where by  $m\Z_{mn}$ we mean the subgroup of index $m$ of $\Z_{mn}$. The vertices of $\Kmn$
will be always understood as elements of $\Z_{mn}$ and the parts of
$\Kmn$ are the cosets of   $m\Z_{mn}$ in $\Z_{mn}$.

Given a graph $\G$ with vertices in $G$ the \emph{list of differences} from $\G$ is the multiset so defined:
$$\Delta \G=\{\pm(x-y)\mid [x,y]\in E(\G)\}.$$
More generally, given a collection $\gamma$ of graphs with vertices in $G$, by $\Delta\gamma$ we denote the union
(counting multiplicities) of all multisets $\Delta\Gamma$ with $\Gamma\in\gamma$.
If $\Gamma$ is a $2$-regular graph, i.e. a disjoint union of cycles, we can introduce also the list of partial differences.
For this purpose the following notation will be useful. Let $c_0,c_1,\ldots,c_{r-1},x$ be elements of $G$, with $x$ of order $d$.
 The closed trail
$$C=[c_0,c_1,c_2,\ldots,c_{r-1},$$
$$c_0+x,c_1+x,c_2+x,\ldots,c_{r-1}+x,\ldots,$$
$$c_0+(d-1)x,c_1+(d-1)x,c_2+(d-1)x,\ldots,c_{r-1}+(d-1)x]$$
will be denoted by
\begin{equation}
\nonumber [c_0,c_1,\ldots,c_{r-1}]_x.
\end{equation}
Also, we set
\begin{eqnarray*}
\phi(C)&=&\{c_0,c_1,c_2,\ldots,c_{r-1}\},\\
\partial C &=& \{\pm(c_{h+1}-c_{h})\ |\ 0\leq h \leq r-2\}\cup\{\pm(c_0+x-c_{r-1})\}.
\end{eqnarray*}
The multiset $\partial C$ is called
the \emph{list of partial differences} from $C$.\\
More generally, given a collection $\mathcal{C}$ of cycles
 with vertices
in $G$, by $\phi(\mathcal{C})$ and $\partial \mathcal{C}$ one means the union (counting
multiplicities) of all multisets $\phi(C)$ and $\partial C$ respectively,
where  $C\in \mathcal{C}$.\\
Notice that the closed trail  $[c_0,c_1,\ldots,c_{r-1}]_x$ is a cycle if and only if the elements
$c_i$, for $i=0,\ldots,r-1$, belong to pairwise distinct cosets of the subgroup $\langle x \rangle$ of $G$.

Now, let $C=(c_0,c_1,\ldots,c_{\ell-1})$ be an $\ell$-cycle with vertices in $G$.
The \emph{stabilizer} of $C$ under the action of $G$ is the subgroup $Stab(C)$ of $G$
defined by $Stab(C)=\{g \in G: C+g=C\}$, where $C+g$ is the cycle $(c_0+g,c_1+g,\ldots,c_{\ell-1}+g)$.
The $G$-\emph{orbit} of $C$ is the set $Orb(C)$ of all distinct cycles in the collection $\{C+g: g\in G\}$.
Then $\phi(C)=\{c_0,c_1,\ldots,c_{\frac{\ell}{d}-1}\}$ and
 $\partial C = \{\pm(c_{h+1}-c_{h})\ |\ 0\leq h < \ell/d\}$, where $d=|Stab(C)|$.
Note that if $d=1$ then $\phi(C)=V(C)$ and $\partial C=\Delta C$.

As well as many results about regular $1$-factorizations of the complete graph have been obtained using the notion of a starter,
see for instance \cite{B,PP,R}, in this paper we obtain cyclic $2$-factorizations
of $\Kmn$ using a generalization of the concept of a $2$-starter introduced in \cite[Definition 2.3]{BR}.

\begin{defini}\label{2starter}
Let $H$ be a subgroup of $G$.
A $2$-\emph{starter in $G$ relative to $H$} is a collection $\Sigma=\{S_1,\ldots,S_t\}$
of $2$-regular graphs with vertices in $G$ satisfying the following conditions:
\begin{itemize}
\item[(\rm{a})] $\partial S_1 \cup \ldots \cup \partial S_t=G\setminus H$;
\item[(\rm{b})] $\phi(S_i)$ is a left transversal of some subgroup $H_i$ of $G$ containing the stabilizers of all cycles of
$S_i$, for $i=1,\ldots,t$.
\end{itemize}
\end{defini}

Note that if $H=\{0\}$ we find again the definition given in \cite{BR}.\\
The importance of investigating the existence of $2$-starters is explained
 in the following theorem.

\begin{thm}\label{thm:2starter}
Let $H$ be a subgroup of $G$.
The existence of a $G$-regular $2$-factoriza\-tion of $K_{|G:H| \times |H|}$ is equivalent to the existence of a $2$-starter in $G$ relative to $H$.
\end{thm}

We omit the proof of the previous theorem since it is an easy generalization of  \cite[Theorem 2.4]{BR}.
In the following, we will apply Theorem \ref{thm:2starter} for $G=\Z_{mn}$ and $H=m\Z_{mn}$.

\section{Necessary conditions for the existence of cyclic $C_\ell$-factorizations}\label{sec3}

First of all, we recall that a trivial necessary condition
for the existence of a $2$-factorization of $\Kmn$ is that
the degree $(m-1)n$ of each vertex  is even.
Also, since we have to organize the $\ell$-cycles into $2$-factors,
$\ell$ must divide $|V(\Kmn)|=mn$.
\begin{rem}\label{rem:neven}
From our hypothesis $\ell$ even, by the above trivial conditions it follows that
$n$ must be even. 
\end{rem}

On the other hand, since we are looking for \emph{cyclic} $C_\ell$-factorizations we need  also
the necessary conditions for the existence of a cyclic $\ell$-cycle system of $\Kmn$, given in \cite{MPP}.

\begin{thm}\cite[Theorem 3.3]{MPP}\label{thm:cn}
Let $n$ be an even integer. A cyclic $\ell$-cycle system of $\Kmn$ cannot exist
in each of the following cases:
\begin{itemize}
\item[(\rm{a})] $m\equiv 0\pmod 4$ and $|\ell|_2=|m|_2+2|n|_2-1$;
\item[(\rm{b})] $m\equiv 1\pmod 4$ and $|\ell|_2=|m-1|_2+2|n|_2-1$;
\item[(\rm{c})] $m\equiv 2,3\pmod 4$, $n\equiv 2\pmod 4$ and $\ell \not\equiv 0 \pmod 4$;
\item[(\rm{d})] $m\equiv 2,3\pmod 4$, $n\equiv 0\pmod 4$ and $|\ell|_2 = 2|n|_2$;
\end{itemize}
where, given a positive integer
$x$, $|x|_2$ denotes the largest $e$ for which $2^e$ divides $x$.
\end{thm}

If the cycles of the system are all hamiltonian, that is $\ell=mn$, we obtain the following corollary.

\begin{cor}\label{cor:ne}
Let $n$ be an even integer. A cyclic hamiltonian $2$-factorization of $\Kmn$ cannot exist
if both $m\equiv 0,3 \pmod 4$ and $n\equiv 2\pmod 4$.
\end{cor}

We remark that if $\ell=4$ from Theorem \ref{thm:cn} we do not obtain non-existence results.
On the other hand we prove the following.

\begin{prop}\label{nonex}
For any odd prime $p$ and any integer $\alpha\geq 1$, a cyclic $C_4$-factorization of $K_{p^\alpha\times 4}$
cannot exist.
\end{prop}

\begin{proof}
In view of Theorem \ref{thm:2starter} we prove that a $2$-starter in $\Z_{4p^\alpha}$ relative to $p^\alpha \Z_{4p^\alpha}$ does not exist. By way of contradiction, suppose the existence of such a $2$-starter $\Sigma$.
Consider $d=4p^{\alpha-1}$. By  item (a) of Definition \ref{2starter}, there exists $S\in \Sigma$ such that $d\in \partial S$. In particular, $d\in \partial C$ for a suitable $4$-cycle $C$ of $S$.
Also, by item (b) of the same definition, $\phi(S)$ is a transversal for some subgroup $H$ of $\Z_{4p^\alpha}$ containing the stabilizers of all cycles of $S$.

Note that a $4$-cycle of type $[0]_t$ cannot exist, otherwise $t=p^\alpha$, in contradiction with item (a) of Definition \ref{2starter}. We point out that this implies that $|\phi(S)|$ has to be even.
Since $d\in \partial C$, we may suppose that there are two vertices $x,y\in\phi(S)$ such that $x-y=d$.
Recalling that $\phi(S)$ is a transversal of $H$, we obtain that $d\not \in H$.

Now, $H$ is a cyclic group, so suppose $H=\langle h\rangle$. Hence $h\nmid d=4p^{\alpha-1}$ and so we have three possibilities:
$h=p^\alpha$ and $|H|=4$; $h=2p^\alpha$ and $|H|=2$; $h=4p^\alpha$ and $|H|=1$. If $|H|\leq 2$, then  $|\phi(S)|\geq 2p^\alpha$ which implies the absurd $|\partial S|\geq 4p^\alpha$.
So $|H|=4$ and $|\phi(S)|=p^\alpha$, which is impossible because, as previously remarked, $|\phi(S)|$ must be even.
\end{proof}

\begin{cor}\label{cor:neC4}
If there exists a cyclic $C_4$-factorization of $\Kmn$ then the following conditions hold:
\begin{itemize}
\item[(\rm{a})] $n$ is even;
\item[(\rm{b})] $n\equiv 0\pmod4$ if $m$ is odd;
\item[(\rm{c})] $n\neq 4$ if $m=p^\alpha$ for some odd prime $p$.
\end{itemize}
\end{cor}
\noindent
%We will show that these conditions are also sufficient.

\section{Auxiliary sets and notation}\label{sec4}

In order to give a uniform construction of cyclic $C_4$-factorizations of $\Kmn$ for any 
choice of admissible values of $m$ and $n$,
it is convenient to fix some notations and define auxiliary sets we will use in the present paper.

Given  $k\in \Z$ we denote by $\ii{k}$ the set
$$\ii{k}=\left\{\begin{array}{cl}
\{0,1,2,\ldots,k\}\subset \N & \textrm{ if } k\geq 0,\\
\emptyset & \textrm{ if } k<0.
\end{array}\right.$$
Furthermore, if $a,b\in \Z$ and $I\subseteq \Z$,
then $a+bI$ denotes the subset $\{a+bi: i \in I \}$.
According to this notation, when $I=\emptyset$, then $a+bI=\emptyset$.
For instance, $2+5\ii{3}$ is the arithmetic progression $\{2,7,12,17\}$.
\medskip

Our method for constructing cyclic $C_4$-factorizations of $\Kmn$ is the following. 
We first construct sets $L=\{[a_i,b_i]: i\in I_L\}$ of edges with $a_i,b_i\in \ii{\frac{mn}{4}-1}$ in such a way that the largest possible number of even integers in $\ii{\frac{mn}{4}}\setminus m\Z$ appears in exactly one set $\Delta L=\{\pm (a_i-b_i): i\in I_L\}$.
Starting from the sets of edges previously described, we construct $2$-regular graphs
$\s L$, whose connected components are $4$-cycles, as follows. Let
$$\s L= \bigcup_{i\in I_L}[a_i,b_i]_{a_i+\frac{mn}{2}}$$
and define $\phi(L)=\phi(\s L)$.
Notice that
$$\partial(\s L)=(\Delta L) \sqcup\left(\frac{mn}{2}-\Delta L\right).$$
Then, we deal with the remaining even integers (if they exist).
Next, we consider the
odd integers in $\ii{\frac{mn}{4}-1}\setminus m\Z$, not appearing in the sets of partial differences of the $2$-regular graphs, previously constructed.
More precisely, given a subset $\D\subseteq \ii{\frac{mn}{4}-1}\setminus m\Z$ consisting of odd integers, we define for all $d\in \D$, the graph
$S_d=[0,d]_{\frac{mn}{2}}$. Observe that $\partial S_d=\pm \left\{d, \frac{mn}{2}-d\right\}$ and that $\phi(S_d)$
is a transversal of the subgroup of index $2$ in $\Z_{mn}$.
Finally, we set
$$\D^\d=\left\{ S_d: d \in \D \right\}.$$
We will show that the set $\Sigma$ of the constructed $2$-regular graphs
 is a $2$-starter in $\Z_{mn}$ relative to $m\Z_{mn}$.
The existence of a cyclic $C_4$-factorization of $\Kmn$ follows from Theorem \ref{thm:2starter}.

\section{Cyclic $C_4$-factorizations of $\Kmn$, $m$ even}\label{sec5}

In this section, we construct cyclic $C_4$-factorizations of $\Kmn$ for $m$ even.
We recall  that also $n$ must be even.

It is  convenient to fix some notation we will use in the next propositions.
For any $k,v\in \N$  let
\begin{eqnarray*}
\Y_k(2,I) & = & \emptyset,\\
\Y_k(2v+2,I) & =&  \{[v-1-i, v+1+i+km] : i \in \ii{v-1}\setminus I\}\quad \textrm{ if } v\geq 1,\\
\Y_k(2v+1,I) & =& \{[v-i,v+1+i+km]: i \in \ii{v}\setminus I \}.
\end{eqnarray*}
It is easy to see that
\begin{eqnarray*}
\Delta \Y_k(2v+2,I)  & =&   \{2+2i+km: i\in \ii{v-1}\setminus I\},\\
\Delta \Y_k(2v+1,I) & = &\{1+2i+km: i\in \ii{v}\setminus I\},\\
\phi(\Y_0(2v+2,\emptyset))& =& \ii{2v}\setminus\{v\},\\
\phi(\Y_0(2v+1,\emptyset))& =& \ii{2v+1}.
\end{eqnarray*}
Also, for  fixed $k\geq 0$ and $v\geq 1$, let
\begin{eqnarray*}
\Y_k(4v) & =  &  \Y_k(4v,\{0\}) \cup \{[2v-1,2v+km],[2v-2,4v-1+km]\};\\
\Y_k(4v+2) & = & \Y_k(4v+2,\emptyset)\cup \{[2v,4v+1+km]\}.
\end{eqnarray*}
Hence, if $m\equiv 0 \pmod 4$ then
$$\Delta\Y_k(m)=\pm \left(4+km+2\ii{\frac{m}{2}-3}\right) \cup \pm \left\{1+km, 1+\frac{m}{2}+km\right\}$$
and if $m\equiv 2 \pmod 4$, $m>2$, then
$$\Delta \Y_k(m)=\pm\left(2+km+2\ii{\frac{m}{2}-2}\right) \cup \pm\left\{\frac{m}{2}+km\right\}.$$
In both cases, $\phi(\Y_k(m))$ is a transversal of the subgroup of index $m$ of $\Z_{mn}$.

Finally, define also the $4$-cycle
$$\ZZ_k = \left(0,2+km,-1,\frac{mn}{2}-3-km \right).$$
Notice that  $\partial  \ZZ_k=\Delta \ZZ_k=\pm \left\{2+km, 3+km, \frac{mn}{2}-2-km, \frac{mn}{2}-3-km\right\}$ and that,
when $m\equiv0\pmod4$,
$\phi(\ZZ_k)$ is a transversal of the subgroup of index $4$ of $\Z_{mn}$.

\begin{lem}\label{bipar}
For all even integers $n$ there exist a cyclic $C_4$-factorization of $K_{2\times n}$
 and one of  $K_{4\times n}$.
\end{lem}

\begin{proof}
Consider $K_{2\times n}$.
If $n\equiv 0\pmod 4$, take $\D=1+2\ii{\frac{n-4}{4}}$ and $\Sigma=\D^\d$.
If $n\equiv 2\pmod 4$, take $\D=1+2\ii{\frac{n-6}{4}}$ and $\Sigma=\left\{[0]_{\frac{n}{2}}\right\} \cup \D^\d$.
It is easy to see that in both cases $\Sigma$ is a $2$-starter in $\Z_{2n}$ relative to $2\Z_{2n}$.

\noindent Consider $K_{4\times n}$.
If $n\equiv 0 \pmod 4$, let $\D=1+4\ii{\frac{n-4}{4}}$ and
$\Sigma=\left\{\ZZ_k: k \in \ii{\frac{n-4}{4}}\right\} \cup \D^\d$.
If $n\equiv 2\pmod 4$, let $\D=1+4\ii{\frac{n-2}{4}}$ and
$\Sigma=\{[0]_n\}\cup \{\ZZ_k : k \in \ii{\frac{n-6}{4}}\} \cup  \D^\d$.
In both cases, $\Sigma$ is a $2$-starter in $\Z_{4n}$ relative to $4\Z_{4n}$.
\end{proof}

\begin{prop}\label{prop:mnpariC4}
For all even integers $m, n$, there exists a cyclic $C_4$-factorization of $\Kmn$.
\end{prop}

\begin{proof}
We split the proof into five subcases according to the congruence class of $m$ modulo $8$ and $n$ modulo $4$.
In view of Lemma \ref{bipar} we may assume $m>4$.
\smallskip

\noindent \underline{Case 1:} $m,n\equiv 0 \pmod 4$.  Let $m\geq 8$ and $n=4t+4$.
Let $\D$ be the set of the odd integers $d\in \ii{\frac{mn}{4}-1}\setminus \{1+km,3+km,1+\frac{m}{2}+km: k \in \ii{t} \}$ and define
$$\Sigma=  \{\ZZ_k,\; \s \Y_k(m): k \in \ii{t}\}\cup \D^\d.$$

\noindent \underline{Case 2:} $m\equiv 2\pmod 4$ and $n\equiv 0 \pmod 4$. Let $m\geq 6$ and $n=4t+4$.
Let $\D$ be the set of the odd integers $d\in \ii{\frac{mn}{4}-1}\setminus \left(\frac{m}{2}+m\ii{t}\right)$ and define
$$\Sigma=\left\{\s \Y_k(m): k \in \ii{t}\right\}\cup \D^\d.$$

\noindent \underline{Case 3:} $m,n\equiv 2\pmod 4$. Let $m\geq 6$ and $n=4t+2$. Let
$S_{0}=\s \Y_{t}(1+\frac{m}{2},\emptyset) \cup [\frac{m-2}{4}]_{\frac{m-2}{4}+\frac{mn}{4}}$.
Clearly, $\partial S_0=\pm(2+tm+2\ii{\frac{m-6}{4}})\cup \pm\{\frac{mn}{4}\}$ and $\phi(S_0)$ is a transversal of the subgroup of
index $\frac{m}{2}$ in $\Z_{mn}$.
Let $\D$ be the set of the odd integers $d\in \ii{\frac{mn}{4}-1}\setminus \left(\frac{m}{2} +m\ii{t-1}\right)$.
Define
$$\Sigma=\{S_0\}\cup \{\s\Y_k(m): k \in \ii{t-1}\}\cup \D^\d.$$

\noindent \underline{Case 4:} $m\equiv 0\pmod 8$ and $n\equiv 2\pmod 4$.
Let $n=4t+2$.
Suppose firstly $m=8$. In this case, let $\D=(7+8\ii{t-1})\cup \{8t+1\}$ and define
$$\Sigma=\left\{[0]_{2n}\right\} \cup \left\{\ZZ_k: k \in \ii{t} \right\}\cup \left\{\s\Y_k(8):  k \in
\ii{t-1}\right\}\cup \D^\d.$$
Now, suppose $m\geq 16$. Let $S_{0}=[0]_{\frac{mn}{4}}$ and $S_2=\s \Y_t \left(\frac{m}{2}\right)$.
For $k\in \ii{t}$ take $S_4(k)=\ZZ_k$ and for  $k\in \ii{t-1}$ take
$S_6(k)=\s \Y_k(m)$. Observe that  $\phi(S_2)$ is a transversal of the subgroup of
 index $\frac{m}{2}$  of $\Z_{mn}$ and that $\Delta \Y_t \left(\frac{m}{2}\right) =\pm(4+mt+2\ii{\frac{m}{4}-3})\cup \pm\{1+mt, 1+\frac{m}{4}+mt\}$.
Let $\D$ be the set of the odd integers in $\ii{\frac{mn}{4}-1}$ not belonging to $\partial S_i$.
Define
$$\Sigma=\{S_0,S_2\} \cup \{S_4(k): k \in \ii{t}\}\cup \{S_6(k): k \in \ii{t-1}\} \cup \D^\d.$$

\noindent \underline{Case 5:} $m\equiv 4\pmod 8$ and $n\equiv 2\pmod 4$.
Let $m\geq 12$ and $n=4t+2$.
Take  $S_{0}=[0]_{\frac{mn}{4}}$ and $S_{2}=\s\Y_{t}\left(\frac{m}{2}\right)$.
Also, for any $k\in \ii{t-1}$ let $S_4(k)=\ZZ_k$ and  $S_6(k)=\s \Y_k(m)$.
Observe that $\phi(S_2)$ is a transversal of the subgroup  of index $\frac{m}{2}$  of $\Z_{mn}$
and that $\Delta \Y_{t}\left(\frac{m}{2}\right)=\pm(2+mt+2\ii{\frac{m}{4}-2})\cup \pm\{\frac{m}{4}+mt\}$.
Let $\D$ be the set of the odd integers in $\ii{\frac{mn}{4}-1}$ not belonging to $\partial S_i$.
Define
$$\Sigma=\{S_0,S_2\} \cup \{S_4(k),S_6(k): k \in \ii{t-1}\}\cup \D^\d.$$

In all these five cases, $\bigcup_{S \in \Sigma}\partial S=\Z_{mn}-m\Z_{mn}$ and hence $\Sigma$ is a $2$-starter in $\Z_{mn}$ relative to $m\Z_{mn}$.
\end{proof}

\begin{ex}
Here, for each case of Proposition \ref{prop:mnpariC4}, following its proof we construct the 
sets necessary to obtain the $2$-starter for a choice of $m$ and $n$.\\
\noindent \underline{Case 1:} Let $m=16$ and $n=8$, hence $t=1$.
We obtain $\D=(1+2\ii{15})\setminus\{1,3,9,17,19,25\}$ and
$$\ZZ_0=(0,2,-1,61),\qquad\ZZ_1=(0,18,-1,45),$$
$$\Y_0(16)=\{[5,9],[4,10],[3,11],[2,12],[1,13],[0,14]\}\cup \{[7,8],[6,15]\},$$
$$\Y_1(16)=\{[5,25],[4,26],[3,27],[2,28],[1,29],[0,30]\}\cup\{[7,24],[6,31]\}.$$
\noindent \underline{Case 2:} Let $m=6$ and $n=12$,  hence  $t=2$.
We obtain $\D=(1+2\ii{8})\setminus\{3,9,15\}$ and
$$\Y_0(6)=\{[1,3],[0,4]\}\cup \{[2,5]\},\qquad \Y_1(6)=\{[1,9],[0,10]\}\cup\{[2,11]\},$$
$$\Y_2(6)=\{[1,15],[0,16]\}\cup\{[2,17]\}.$$
\noindent \underline{Case 3:} Let $m=10$ and $n=6$,  hence $t=1$.
We obtain
$$S_0=\{[1,13]_{31},[0,14]_{30},[2]_{17}\},\quad \Y_0(10)=\{[3,5],[2,6],[1,7],[0,8]\}\cup\{[4,9]\}$$
and  $\D=(1+2\ii{6})\setminus\{5\}$.\\
\noindent \underline{Case 4:} Let $m=32$ and $n=6$,  hence $t=1$. We obtain $S_0=[0]_{48}$,
$$\Y_1(16)=\{[5,41],[4,42],[3,43],[2,44],[1,45],[0,46]\}\cup\{[7,40],[6,47]\},$$
$$\ZZ_0=(0,2,-1,93),\qquad \ZZ_1=(0,34,-1,61),$$
$$\Y_0(32)=\{[13,17],[12,18],[11,19],[10,20],[9,21],[8,22],[7,23],[6,24],[5,25],$$
$$[4,26],[3,27],[2,28],[1,29],[0,30]\}\cup \{[15,16],[14,31]\}.$$
Hence $\D=(1+2\ii{23})\setminus\{1,3,17,33,35,41\}$.\\
\noindent \underline{Case 5:} Let $m=12$ and $n=6$,  hence $t=1$. We obtain $S_0=[0]_{18}$, $$\Y_1(6)=\{[1,15],[0,16]\}\cup\{[2,17]\},\qquad \ZZ_0=(0,2,-1,33),$$
$$\Y_0(12)=\{[3,7],[2,8],[1,9],[0,10]\}\cup \{[5,6],[4,11]\}.$$
Hence  $\D=(1+2\ii{8})\setminus\{1,3,7,15\}$.
\end{ex}

\section{Cyclic $C_4$-factorizations of $\Kmn$, $m$ odd}\label{sec6}

In this section, we construct cyclic $C_4$-factorizations of $\Kmn$ for $m$ odd.
By Corollary \ref{cor:neC4} we have that $n\equiv 0\pmod 4$.

In order to explain our constructions, we need to define some other auxiliary sets.
For any $v,k\in\N$ and any  subset $I\subset \N$ we define
\begin{eqnarray*}
\X_k^{e}(4v+2,I) & =& \{ [2v-2-2i,2v+2i+km] : i\in \ii{v-1}\setminus I\},\\
\X_k^{o}(4v+2,I) & =& \{ [2v-1-2i,2v+3+2i+km] : i\in \ii{v-1}\setminus I \},\\
\X_k^e(4v+4,I) & =& \{ [2v-2-2i,2v+2+2i+km] : i\in \ii{v-1}\setminus I \},\\
\X_k^o(4v+4,I) & =& \{ [2v+1-2i,2v+3 +2i+km]  : i\in \ii{v}\setminus I\}.
\end{eqnarray*}
Notice that
\begin{eqnarray*}
\Delta \X_k^e(4v+2,I) & =& \pm\{2+4i+km: i\in \ii{v-1}\setminus I\},\\
\Delta \X_k^o(4v+2,I) & =& \pm\{4+4i+km: i\in \ii{v-1}\setminus I \},\\
\Delta \X_k^e(4v+4,I) & =& \pm\{4+4i+km: i\in \ii{v-1}\setminus I\},\\
\Delta \X_k^o(4v+4,I) & =& \pm\{2+4i+km: i\in \ii{v}\setminus I \},
\end{eqnarray*}
and that, for $k=0$ and $I=\emptyset$,
\begin{eqnarray}
\phi\left(\X_0^e(4v+2,\emptyset)\cup \X_0^o(4v+2,\emptyset)\right) & =&\ii{4v+1}\setminus\{2v+1, 4v\},\label{chi2}\\
\phi\left(\X_0^e(4v+4,\emptyset)\cup \X_0^o(4v+4,\emptyset)\right) & = &\ii{4v+3}\setminus \{2v, 4v+2\}\label{chi4}.
\end{eqnarray}

\begin{prop}\label{prop:modd}
For any $m$ odd and any $n\equiv0 \pmod 8$ there exists a cyclic $C_4$-factorization of $\Kmn$.
\end{prop}

\begin{proof}
We split the proof into two subcases according to the congruence class of $n$ modulo $16$.
\smallskip

\noindent \underline{Case 1:} $n\equiv 0 \pmod{16}$.
Let $n=16t+16$, whence $\frac{mn}{4}=4m(t+1)=4v+4$, i.e. $v=m(t+1)-1$.
We start removing from the sets $\X_0^e(\frac{mn}{4},\emptyset)$ and $\X_0^o(\frac{mn}{4},\emptyset)$ the edges $[a,b]$ such that $b-a$ is a multiple of $m$.
First, take  $\X_0^e(\frac{mn}{4},\emptyset)$: we have to remove the edges giving the difference $4+4i=\alpha m$ for some $\alpha\in \N$. Clearly $\alpha\equiv 0 \pmod 4$
and so, setting $\alpha=4j+4$, from $4+4i=(4j+4)m$ we get $i=m-1+jm$.
Since $i \in \ii{v-1}$, it turns out that we have to remove the edges $[2v-2-2i, 2v+2+2i]$  with  $i\in I=m-1+m\ii{\left\lfloor\frac{v-m}{m}\right\rfloor}=m-1+m\ii{t-1}$.
So, the vertices we removed are the elements of the set $(2m-2+2m\ii{t-1})\cup (2mt+4m-2+2m\ii{t-1})$ which coincides
with the set $(2m-2+2m\ii{2t})\setminus \{2mt+2m-2\}$.

Now, take $\X_0^o(\frac{mn}{4},\emptyset)$: we remove the edges giving the difference $2+4i=\alpha m$ for some $\alpha\in \N$. Clearly $\alpha\equiv 2 \pmod 4$.
Setting $\alpha=4j+2$, from $2+4i=(4j+2)m$ we get $i=\frac{m-1}{2}+jm$.
Since $i \in \ii{v}$, it turns out that we have to remove the edges $[2v+1-2i, 2v+3+2i]$  with
$i\in I=\frac{m-1}{2}+m\ii{\left\lfloor\frac{2v-m+1}{2m}\right\rfloor}=\frac{m-1}{2}+m\ii{t}$.
So, the vertices we removed are the elements of the set $(m+2m\ii{t})\cup (2mt+3m+2m\ii{t})=m+2m\ii{2t+1}$.

Let
$$
\X(4v+4) = \X_0^e(4v+4, m-1+ m\ii{t-1})\cup\,\X_0^o\left(4v+4, \frac{m-1}{2}+m\ii{t}\right).
$$
Then $$\partial\X(4v+4)=\pm\left(2\ii{\frac{mn}{8}-1}\setminus 2m\ii{\frac{n}{8}-1}\right),$$
and, recalling \eqref{chi4},
$$\phi\left(\X\left(4v+4\right)\right)=\ii{\frac{mn}{4}-1}\setminus
((2m-2+2m\ii{2t+1})\cup (m+2m\ii{2t+1})).$$
Taking
\begin{eqnarray*}
A&=&\{[ (1+2t-2j)m, 2m(t+2+j)-2]: j \in \ii{t}\}\cup \\
&&\{ [2m(t+1-j)-2, m(2t+3+2j)]: j \in \ii{t}  \}
\end{eqnarray*}
we obtain
$$\partial A=\pm\left(m+2+4m\ii{t}\right)\cup \pm\left(3m-2+ 4m\ii{t}\right)$$
and
$$\phi(A)=(2m-2+2m\ii{2t+1}) \cup (m+2m\ii{2t+1}).$$
Let now  $B=\X\left(4v+4\right)\cup A$,
whence $\phi(B)=\ii{\frac{mn}{4}-1}$.

Let $\D$ be the set of the odd integers $d\in \ii{\frac{mn}{4}-1}\setminus \left(m\Z\cup \partial A\right)$.
Define $S_0=\s B$ and
 $\Sigma=\{S_0\}\cup \D^\d$. It is easy to see that $\Sigma$ is a $2$-starter in $\Z_{mn}$
 relative to $m\Z_{mn}$.
Indeed $\partial S_0\cup \left(\cup_{d\in \D}\partial S_d \right)=\Z_{mn}\setminus m\Z_{mn}$ and
$\phi(S_0)$ is a transversal of the subgroup of index $\frac{mn}{4}$ of $\Z_{mn}$.
\smallskip

\noindent \underline{Case 2:} $n\equiv 8\pmod{16}$.  Let $n=16t+8$, whence $\frac{mn}{4}=2(2t+1)m=4v+2$,
 i.e. $v=tm+\frac{m-1}{2}$.
As before we start removing from the subsets
$\X_0^e(\frac{mn}{4},\emptyset)$ and $\X_0^o(\frac{mn}{4},\emptyset)$ the edges $[a,b]$ such that $b-a$ is a multiple of $m$.
First, take  $\X_0^e(\frac{mn}{4},\emptyset)$: we have to remove the edges giving the difference $2+4i=\alpha m$ for some $\alpha\in \N$. Clearly $\alpha\equiv 2 \pmod 4$
and so, setting $\alpha=4j+2$, from $2+4i=(4j+2)m$ we get $i=\frac{m-1}{2}+jm$.
Since $i \in \ii{v-1}$, it turns out that we have to remove the edges
$[2v-2-2i,2v+2i]$ with  $i\in I=\frac{m-1}{2}+m
\ii{\left\lfloor\frac{2v-m-1}{2m}\right\rfloor}=\frac{m-1}{2}+m\ii{t-1}$.
So, the vertices we removed are the elements of the set
$(2m-2+ 2m\ii{t-1})\cup (2m-2+2mt+2m\ii{t-1})$ which coincides
with the set $(2m-2+ 2m\ii{2t-1})$.

Now take $\X_0^o(\frac{mn}{4},\emptyset)$: we remove the edges giving the difference $4+4i=\alpha m$ for some $\alpha\in \N$. Clearly $\alpha\equiv 0 \pmod 4$.
Setting $\alpha=4j+4$, from $4+4i=(4j+4)m$ we get $i=m-1+jm$.
Since $i \in \ii{v-1}$, it turns out that we have to remove the edges $[2v-1-2i,2v+3+2i]$   with
$i\in I=m-1+m\ii{\left\lfloor\frac{v-m}{m}\right\rfloor}=m-1+m\ii{t-1}$.
So, the vertices we removed are the elements of the set
$(m+2m\ii{t-1})  \cup (2mt+3m+2m\ii{t-1})=(m+2m\ii{2t})\setminus \{m+2mt\}$.

Let
$$
\X(4v+2) = \X_0^e\left(4v+2, \frac{m-1}{2}+ m \ii{t-1}\right)
\cup\, \X_0^o(4v+2, m-1 +m\ii{t-1}).
$$
Then $$\partial\X(4v+2)=\pm\left(2\ii{\frac{mn}{8}-1}\setminus 2m\ii{\frac{n}{8}-1}\right),$$
and, recalling \eqref{chi2},
$$\phi\left(\X\left(4v+2\right)\right)=\ii{\frac{mn}{4}-1}\setminus
((2m-2+2m\ii{2t})\cup (m+2m\ii{2t})).$$
Taking
\begin{eqnarray*}
A&=&\{[ m(2t+1-2j), 2m(t+1+j)-2]: j \in \ii{t}\}\cup \\
&&\{ [2m(t-j)-2, m(2t+3+2j)]: j \in \ii{t-1}  \}
\end{eqnarray*}
we obtain 
$$\partial A=\pm\left(m-2+4m\ii{t}\right)\cup \pm\left(3m+2+ 4m\ii{t-1}\right)$$
and
$$\phi(A)=(2m-2+2m\ii{2t}) \cup (m+2m\ii{2t}).$$
Let now $B=\X\left(4v+2\right)\cup A$, whence
 $\phi(B)=\ii{\frac{mn}{4}-1}$.

Let $\D$ be the set of the odd integers $d\in \ii{\frac{mn}{4}-1}\setminus \left(m\Z \cup \partial A\right)$.
Define $S_0=\s B$ and  $\Sigma=\{S_0\}\cup \D^\d$.
Arguing as before one can see that $\Sigma$ is a $2$-starter in $\Z_{mn}$
relative to $m\Z_{mn}$.
\end{proof}

\begin{ex}
Following the proof  of Proposition \ref{prop:modd} we construct the sets 
necessary to obtain the $2$-starter for two choices of $m$ and $n$.\\
\noindent \underline{Case 1:} Let $m=3$ and $n=32$, hence $t=1$ and $v=5$.
We obtain
$$B=\{[8,12],[6,14],[2,18],[0,20]\}\cup\{[11,13],[7,17],[5,19],[1,23]\}\cup$$
$$\{[9,16],[3,22]\}\cup\{[10,15],[4,21]\},$$
which implies $\D=\{1,11,13,23\}$.\\
\noindent \underline{Case 2:} Let $m=5$ and $n=40$, hence $t=2$ and $v=12$. We obtain
$$B=\{[22,24],[20,26],[16,30],[14,32],[12,34],[10,36],[6,40],[4,42],[2,44],[0,46]\}\cup$$
$$\{[23,27],[21,29],[19,31],[17,33],[13,37],[11,39],[9,41],[7,43],[3,47],[1,49]\}\cup$$
$$\{[25,28],[15,38],[5,48]\}\cup\{[18,35],[8,45]\},$$
which gives  $\D=(1+2\ii{24})\setminus\{3,5,15,17,23,25,35,37,43,45\}$.
\end{ex}

\begin{lem}\label{15}
There exist a cyclic $C_4$-factorization of $K_{15\times 4}$ and one of $K_{35\times 4}$.
\end{lem}

\begin{proof}
A $2$-starter $\Sigma$ in $\Z_{60}$ relative to $15\Z_{60}$ is given by
$\Sigma=\{S_0,S_2,S_4\}\cup \D^\d$, where
$S_0=\s\{[0,8], [1,7], [2,6], [3,15], [4,9]\}$,
$S_2=(0,10,9,29)\cup [1,8]_{31}$,
$S_4=\s\{[0,2],[1,15]\}$ and $\D=\{3,9,11,13\}$.

A $2$-starter $\Sigma$ in $\Z_{140}$ relative to $35\Z_{140}$ is given by
$\Sigma=\{S_0,S_2,S_4,S_6,S_8\}\cup \D^\d$, where
$S_0=\s\{[0,12], [1,11], [2,10], [3,9],  [4,8], [5,7], [6,13]\}$,
$S_2=\s\{[0,26], [1,25],$ $[2,24], [3,23], [4,36], [5,21], [6,27]\}$,
$S_4=\s\{[0,28], [1,15], [2,19], [3,16], [4,7]\}$,
$S_6=(0,34,33,69) \cup [1,28]_{71}\cup [2,17]_{72}\cup [5,16]_{75}$,
$S_8$ $=\s\{[0,18], [1,31]\}$, and $\D=\{5,9,19,23,25,29,31,33\}$.
\end{proof}

\begin{prop}\label{prop:mx4}
Let $m$ be an odd integer.
There exists a cyclic $C_4$-factorization of $K_{m\times 4}$ if and only if
$m\neq p^\alpha$ where $p$ is a prime.
\end{prop}

\begin{proof}
In Proposition \ref{nonex} we have already proved that $m\neq p^\alpha$ is a 
necessary condition for the existence of a cyclic $C_4$-factorization of $K_{m\times 4}$.
We now prove the sufficiency.

Firstly, write $m$ as a product $ab$ with $a<b$ and $\gcd(a,b)=1$.
In view of Lemma \ref{15} we can assume $m\neq 15,35$.
Now, for all $i\in \ii{\frac{a-3}{2}}$, take the set
$$W_i=\{[b-2-j, b+j+2bi]: j \in \ii{b-2}\}\cup \{[b-1,2b-1+2bi]\}.$$
Observe that $\Delta W_i=\pm (2+2bi+2\ii{b-2})\cup \pm \{2b(i+1)\}$ and that $\phi(W_i)$ is a transversal of the subgroup of index $2b$ in $\Z_{4m}$.
In particular,
$$\ccup_{i=0}^{\frac{a-3}{2}} \Delta W_i =\pm\left( 2\ii{\frac{ab-b}{2}}\setminus 2b\Z\right) \cup \pm\left( b+2b\ii{\frac{a-3}{2}}\right).$$
We are left to consider the two sets of even differences
$$A=2b+2b\ii{\frac{a-3}{2}} \quad \textrm{ and} \quad B=b(a-1)+2+2\ii{\frac{b-3}{2}}.$$
Observe that $|A|=\frac{a-1}{2}$ and $|B|=\frac{b-1}{2}$.
Also, the elements of $A$ are pairwise distinct modulo $2a$ and different from $0$ modulo $2a$.

We now have to consider four cases, according to  the congruence class of $a$ and $b$ modulo $4$.
In each of these four cases we start constructing two particular sets of edges in the following way.
Let $C$ be a set of even integers $0<c_k<ab$, pairwise distinct and non-zero modulo $2a$, such that  $N=|C|<a$ is even.
We order such elements according to the following rule.
Consider the euclidean division of $c_k$ by $2a$ and let $r_k$ be its the remainder. Take
$r_{k_0}> r_{k_1}>\ldots>r_{k_{N-1}}$ and consider the corresponding elements
$c_{k_0}, c_{k_1},\ldots, c_{k_{N-1}}$.
Now let $J_1=\{[j, j+c_{k_j}]: j \in \ii{N-1}\}$. Notice that
$\Delta J_1=\pm C$ and that the elements of $\phi(J_1)$ are pairwise distinct modulo $2a$.
The set $\ii{2a-1}\setminus \phi(J_1)$ contains  $a-N$ even integers
$N= x_0<x_1<\ldots < x_{a-N-1}\leq 2a-2$ and $a-N$ odd integers
$N+1\leq y_{a-N-1}<y_{a-N-2}<\ldots< y_1<y_0= 2a-1$.
Define $J_2=\{[x_j, y_j]: j \in \ii{a-N-1}\}$.
For simplicity, write the previous edges in the form $[z_k,w_k]$ with $0<d_k=w_k-z_k<2a$. Clearly, each difference
$d_k$ appears at most twice in $\Delta J_2$. Now, we made the following changes:
\begin{itemize}
\item[($a$)] if $d_{k_1}=d_{k_2}\not\in \{ b, b-2a, 3b-2a\}$, replace in $J_2$ the edge
$[z_{k_1},w_{k_1}]$ with the edge $[z_{k_1},w_{k_1}+2a]$ (so, we get the differences
$d_{k_1}$ and $d_{k_1}+2a$);
\item[($b$)] if $d_{k_1}=d_{k_2}=3b-2a$ (which implies $\frac{2a}{3}<b<\frac{4a}{3}$), replace in $J_2$ the edge
$[z_{k_1},w_{k_1}]$ with the edge $[z_{k_1},w_{k_1}+4a]$ (so, we get the differences
$3b-2a$ and $3b+2a$);
\item[($c_1$)] if $d_k=b$ (which implies $b<2a$) appears once in $\Delta J_2$, replace in $J_2$ the edge
$[z_k,w_k]$ with the edge $[z_k,w_k+2a]$ (so, we get the difference
$b+2a$);
\item[($c_2$)]  if $d_{k_1}=d_{k_2}=b$ (which implies $b<2a$), replace in $J_2$ the edges
$[z_{k_1},w_{k_1}]$ and $[z_{k_2},w_{k_2}]$ with the edges $[z_{k_1},w_{k_1}+2a]$ and $[z_{k_2},w_{k_2}+4a]$ (so, we get the differences
$b+2a$ and $b+4a$);
\item[($c_3$)] if $d_{k_1}=d_{k_2}=b-2a$ (which implies $2a<b<4a$),
replace in $J_2$ the edge $[z_{k_1},w_{k_1}]$ with the edge $[z_{k_1},w_{k_1}+4a]$ (so, we get the differences $b-2a$ and $b+2a$).
\end{itemize}
Notice that at most one of ($c_1$), ($c_2$) and ($c_3$) can happen and that the new differences obtained in this process are all less than $6a<ab$, since $m\neq 15$.
Call $J(C)$ the set of edges obtained from $J_1\cup J_2$ with the previous modifications.
We obtain that $\phi(J(C))$ is a transversal of the subgroup of index $2a$ in $\Z_{4m}$
and that the elements of $\Delta J(C)$ are pairwise distinct and none of them is an odd multiple of $b$.

Now, let $F=\{f_k=(a-1)b+2+2k: k \in K\}$, where $K\subseteq \ii{\frac{b-3}{2}}$ is a set of even cardinality $M$. It follows that $F$ is  a subset of $B$.
Let $W'$ be the set of edges obtained from
$W_{\frac{a-3}{2}}$ by replacing  the edge $[b-2-k,b+k+(a-3)b]$ with the edge $[b-2-k, b+k+(a-1)b]$, whenever  $f_k\equiv 0 \pmod 4$. 
Observe that $f_k\in \Delta W'$ and that $\phi(W')=\phi\left(W_{\frac{a-3}{2}}\right)$ modulo $2b$.
After this modification, we have to consider the differences $f_k-2b$. 
So, we define  $\tilde f_k=f_k-2b$ if $f_k\equiv 0 \pmod 4$ and $\tilde f_k=f_k$  otherwise.
Let $\tilde F$ be the set of such elements $\tilde f_k$, that we order as $\tilde f_{k_0}<\tilde f_{k_1}<\ldots<\tilde f_{k_{M-1}}$. Finally, we construct the set $G(F)=\left\{G_t: t\in \ii{\frac{M}{2}-1}\right\}$, where $G_t=\{[0,\tilde f_{k_{2t}}], [1,1+\tilde f_{k_{2t+1}}]\}$.
Note that $\phi(G_t)$ is a transversal of the subgroup of index $4$ in $\Z_{4m}$ and
that $\Delta G(F)=\pm\tilde F$.
\smallskip

\noindent \underline{Case 1:} $a\equiv b \equiv 1 \pmod 4$.
In this case $|A|$ and $|B|$ are both even.
We take $C=A$ and $F=B$ and construct the sets $J(A)$ and $G(B)$.
A $2$-starter in $\Z_{4m}$ relative to $m\Z_{4m}$ can be obtained considering the set
$$\Sigma=\left\{\s W_i: i \in \ii{\frac{a-5}{2}}\right\}\cup \{\s W', \s J(A), \s G(B)\}\cup \D^\d,$$
where $\D$ is the set of the odd integers in $\ii{m-1}$ not belonging to $\Delta\{W_i,W',J(A),$ $G(B)\}$.
\smallskip

\noindent \underline{Case 2:} $a\equiv b \equiv 3 \pmod 4$.
In this case $|A|$ and $|B|$ are both odd.
Let $\beta$ be an element of $B$ such that $\beta\not \equiv 0,2bi \pmod{2a}$ for all $i\in \ii{\frac{a-3}{2}}$.
We take $C=A\cup \{\beta\}$ and $F=B\setminus \{\beta\}$ and construct the sets $J(C)$ and $G(F)$.
A $2$-starter in $\Z_{4m}$ relative to $m\Z_{4m}$ can be obtained considering the set
$$\Sigma=\left\{\s W_i: i \in \ii{\frac{a-5}{2}}\right\}\cup \{\s W', \s J(C), \s G(F)\}\cup \D^\d,$$
where $\D$ is the set of the odd integers in $\ii{m-1}$ not belonging to $\Delta\{W_i,W',J(C),$ $G(F)\}$. 
Clearly, if $a=3$, we have not to consider the sets $W_i$'s.
\smallskip

\noindent \underline{Case 3:} $a\equiv -b \equiv 1 \pmod 4$.
In this case $|A|$ is even and $|B|$ is odd.
We take $C=A$ and $F=B\setminus \{ab-1\}$ and construct the sets $J(A)$ and $G(F)$.
We need also another set of edges.
Let $q_0<q_1<\ldots<q_{2a-5}$ be the elements of the set $\ii{2a-1}\setminus \{0,a-2,a-1,2a-1\}$.
Let $Q'$ be the set of the edges $[q_j, q_{2a-5-j}+6a]$, where we take 
$[q_j, q_{2a-5-j}+8a]$ instead of $[q_j, q_{2a-5-j}+6a]$, if $q_{2a-5-j}+6a-q_j$ is a multiple of $b$ (this can happen at most once).

If either $1+6a$ or $1+8a$ is a multiple of $b$, consider the cycle
$\Gamma=(0,ab-1, ab-2-10a, 2ab-1-10a)$;
otherwise let $\Gamma=(0,ab-1, ab-2-8a, 2ab-1-8a)$.
Notice that, since $m\neq 35$, we have $1+10a<ab$.
Let $Q$ be the $2$-regular graph $\s Q'\cup \Gamma $.
 Observe that $\phi(Q)$ is a transversal of the subgroup of index $2a$ in $\Z_{4m}$ 
 and that $\Delta Q$  contains the even integer $ab-1$ and is disjoint from $\Delta\{W_i,W',J(A),G(F)\}$.
A $2$-starter in $\Z_{4m}$ relative to $m\Z_{4m}$ can be obtained considering the set
$$\Sigma=\left\{\s W_i: i \in \ii{\frac{a-5}{2}}\right\}\cup \{\s W', \s J(A), \s G(F), Q\}\cup \D^\d,$$
where $\D$ is the set of the odd integers in $\ii{m-1}$ not belonging to $\Delta\{W_i,W',J(A),$ $G(F),Q\}$.
\smallskip

\noindent \underline{Case 4:} $a\equiv -b \equiv 3 \pmod 4$.
In this case $|A|$ is odd and $|B|$ is even.
We take $C=A\setminus \{2b\}$ and $F=B$ and construct the sets $J(C)$  and $G(B)$.
Also in this case, we need another set of edges.
Let $q_0<q_1<\ldots<q_{2a-5}$ be the elements of the set $\ii{2a-1}\setminus \{0,r-1,r,2a-1\}$, where $r$ is the remainder in the euclidean division of $2b$ by $2a$.
Let $Q'$ be the set of the edges $[q_j, q_{2a-5-j}+6a]$,  where we take
$[q_j, q_{2a-5-j}+8a]$ instead of $[q_j, q_{2a-5-j}+6a]$, if $q_{2a-5-j}+6a-q_j$ is a multiple of $b$ (this can happen at most once).

If  $m=63$, i.e. $(a,b)=(7,9)$, let  $\Gamma=(0,18,3,111)$. Suppose now $m\neq 63$:
if either $1+6a$ or $1+8a$ is  a multiple of $b$, let $\Gamma=(0,2b, 2b-1-10a, 2ab-1-10a)$;
otherwise let $\Gamma=(0,2b, 2b-1-8a, 2ab-1-8a)$. Since $m\neq 15,63$,  we obtain $1+10a<ab$.
Let $Q$ be the $2$-regular graph $\s Q'\cup \Gamma$.
Notice that $\phi(Q)$ is a transversal of the subgroup of index $2a$ in $\Z_{4m}$ and that 
$\Delta Q$  contains the even integer $2b$ and is disjoint from $\Delta\{W_i,W',J(C),G(B)\}$.
A $2$-starter in $\Z_{4m}$ relative to $m\Z_{4m}$ can be obtained considering the set
$$\Sigma=\left\{\s W_i: i \in \ii{\frac{a-5}{2}}\right\}\cup \{\s W', \s J(C), \s G(B), Q\}\cup \D^\d,$$
where $\D$ is the set of the odd integers in $\ii{m-1}$ not belonging to $\Delta\{W_i,W',J(C),$ $G(B),Q\}$. Clearly, if $a=3$, we have not to consider the sets $W_i$'s and $J(C)$.
\end{proof}

\begin{ex}
Let $m=143$. Following the proof of Proposition \ref{prop:mx4} we construct the elements of a $2$-starter $\Sigma$ in
$\Z_{572}$ relative to $143\Z_{572}$.
We factorize $m$ as $m=ab$, where $a=11$ and $b=13$ (hence we are in Case 4 of the proof).
For any $i\in \ii{3}$ let
$$W_i=\{[11-j, 13+j+26i]: j \in \ii{11}\}\cup \{[12,25+26i]\}$$
and instead of $W_4$ we consider
\begin{eqnarray*}
W' & =& \{[0,128],[1,127],[2,126],[3,125],[4,124],[5,123],[6,122],[7,147],[8,120],\\
&& [9,145],[10,118],[11,143],[12,129]\}.
\end{eqnarray*}
We obtain $A=\{26,52,78,104,130\}$ and $B=\{132,134,136,138,140,142\}$.
Let $C=A\setminus \{26\}$ and construct
\begin{eqnarray*}
J(C) & =& \{[0,130], [1,105],[2,80],[3,55],[4,21],[6,41],[8,15],[10,35],[9,12],\\
&&[7,16],[5,62]\}.
\end{eqnarray*}
Then we get $G(B)=\{G_0,G_1,G_2\}$, where
$$G_0=\{[0,106],[1,111]\},\quad G_1=\{[0,114],[1,135]\},\quad G_3=\{[0,138],[1,143]\}.$$
Also,
\begin{eqnarray*}
Q &=& \s\{[1,86],[2,85],[5,84],[6,83],[7,82],[8,81],[9,80],[10,79],[11,78]\} \cup \\
& &(0,26,509,197).
\end{eqnarray*}
Finally the elements of $\D$ are the odd integers of $\ii{142}\setminus\{3,7,9,13,17,25,35,39,57,$ $65,67,69,71,73,75,77,79, 83,85,89,91,117\}$.

Hence $\Sigma=\{\s W_0, \s W_1, \s W_2, \s W_3, \s W',\s J(C), \s G(B), Q\}\cup \D^\d$.
\end{ex}

\begin{prop}\label{prop:m1n4}
For any $m\equiv 1\pmod 4$  and any $n\equiv 4 \pmod 8$ with $n>4$, there exists a cyclic $C_4$-factorization of $\Kmn$.
\end{prop}

\begin{proof}
Let  $n=8t+12$. Define $S_0=\s A$ where
\begin{eqnarray*}
A& =&\Y_1(m+1,\emptyset)\cup \left(\frac{m-1}{2}+
\X_{2(t+1)}^e(m+1,\emptyset)\right)\cup\\
&&\left(\frac{m-1}{2}+
\X_{2(t+1)}^o(m+1,\emptyset)\right)\cup \left\{\left[m,\frac{3(m-1)}{2}\right]\right\}.
\end{eqnarray*}
Observe that
$\phi(S_0)=\ii{m-2} \cup \{m,\frac{3m-3}{2}\}\cup (\frac{3m+1}{2}+\ii{\frac{m-3}{2}})\cup
\{m-1+2(t+1)m\}\cup (m+1+2(t+1)m+\ii{\frac{m-7}{2}})\cup \{\frac{3m-1}{2}+2(t+1)m\}$,
and so $\phi(S_0)$ is a transversal of the subgroup of index $2m$ of $\Z_{mn}$.
Also,
$\Delta A= \pm(m+2+2\ii{\frac{m-3}{2}})\cup \pm(2+2(t+1)m+2\ii{\frac{m-3}{2}})\cup\pm\{\frac{m-3}{2}\}$.
Now, for all $k\in \ii{t}$, let $S_2(k)=\s B_k$ where
$$
B_k=\X_{2k}^e(2m,\emptyset) \cup \X_{2k}^o(2m,\emptyset)\cup\{[m,2m-2+2km]\}.
$$
It is easy to see
$\phi(S_2(k))$ is a transversal of the subgroup of index $2m$ of $\Z_{mn}$ and
$\Delta B_k=\pm(2+2km+2\ii{m-2})\cup \pm\{m-2+2km\}$.
Let $\D$ be the set of the odd integers in $\ii{\frac{mn}{4}-1}\setminus m\Z$ not belonging to $\Delta\{A,B_k\}$.
Define
$$\Sigma=\{S_0\} \cup \{S_2(k): k \in \ii{t}\} \cup \D^\d.$$
We get $\ccup_{S \in \Sigma}\partial S=\Z_{mn}-m\Z_{mn}$. Hence $\Sigma$ is a $2$-starter in $\Z_{mn}$ relative to $m\Z_{mn}$.
\end{proof}

\begin{ex}
Let $m=9$ and $n=12$. Following the proof of previous proposition we obtain
$A=\{[3,14],[2,15],[1,16],[0,17]\}\cup\{[6,26],[4,28]\}\cup\{[7,29],[5,31]\}\cup\{[9,12]\}$
and $B_0=\{[6,8],[4,10],[2,12],[0,14]\}\cup\{[7,11],[5,13],[3,15],[1,17]\}\cup\{[9,16]\}$.
Hence, it results $\D=\{1,5,19,21,23,25\}$.
\end{ex}

\begin{lem}\label{lem:37}
For any  $n\equiv 4 \pmod 8$  with $n>4$, there exist a cyclic
$C_4$-factoriza\-tion of $K_{3\times n}$ and one of $K_{7\times n}$.
\end{lem}

\begin{proof}
First consider $K_{3\times n}$ with $n=8t+12$.  For all $k\in \ii{t}$ take
$$S_0(k)=[0,2+6k]_{3n/2}\cup [1, 5+6k]_{3n/2+1}\cup [3,4+6k]_{3n/2+3}.$$
Take also
$$S_2=\left(0,\frac{3n}{4}-1, 1, \frac{3n}{4}+2\right) \cup [4,9]_{3n/2+4}$$
and  $\D=\left(11+6\ii{t-1}\right)\setminus 3\Z$.
Observe that $\phi(S_0(k))$ and $\phi(S_2)$ are transversals of the subgroup of index $6$ in $\Z_{3n}$.
Also, $\partial S_0(k)=\pm \left\{2+6k,4+6k,1+6k, \frac{3n}{2}-2-\right.$ $6k,\left.\frac{3n}{2}-4-6k,\frac{3n}{2}-1-6k\right\}$ and $\partial
S_2=\pm \left\{ \frac{3n}{4}-1,\frac{3n}{4}-2,\frac{3n}{4}+1, \frac{3n}{4}+2, 5,\frac{3n}{2}\right.$
$\left.-5 \right\}$.
The set
$$\Sigma=\{S_0(k): k \in \ii{t}\}\cup\{S_2\}\cup \D^\d$$
is a $2$-starter in $\Z_{3n}$ relative to $3\Z_{3n}$.

Now consider $K_{7\times n}$ with $n=8t+12$. For all $k\in \ii{t}$ take
$$S_0(k)=[0,10+14k]_{7n/2}\cup [1, 13+14k]_{7n/2+1}\cup [2,8+14k]_{7n/2+2}\cup$$
$$[3,11+14k]_{7n/2+3}\cup[4,6+14k]_{7n/2+4}\cup [5, 9+14k]_{7n/2+5}\cup [7, 12+14k]_{7n/2+7}.$$
Take also
$$S_2=
\left(0,\frac{7n}{4}-5, 3, \frac{7n}{4}+8\right)\cup
\left(4, \frac{7n}{4}+3, 5, \frac{7n}{4}+6\right)\cup$$
$$\left(8, \frac{7n}{4}+5, 11, \frac{7n}{4}+14\right)\cup [6,9]_{7n/2+6}.$$
Observe that $\phi(S_0(k))$ and $\phi(S_2)$ are transversals of the subgroup of index $14$ in $\Z_{7n}$.
Also, $\partial S_0(k)=\pm \left\{2+2i+14k,\frac{7n}{2}-2-2i-14k: i \in \ii{5} \right\}\cup \pm\left\{
5+14k, \frac{7n}{2}-5\right.$ $\left.-14k\right\}$ and $\partial S_2=\pm\left\{\frac{7n}{4}\pm 1, \frac{7n}{4}\pm 2, \frac{7n}{4} \pm 3, \frac{7n}{4}\pm 5, \frac{7n}{4}\pm 6, \frac{7n}{4}\pm 8, 3,\frac{7n}{2}-3\right\}$.
Let $\D$ be the set of the odd integers in $\ii{\frac{7n}{4}-1}\setminus 7\Z$ not belonging to the previous sets of differences.
The set
$$\Sigma=\{S_0(k): k \in \ii{t}\}\cup\{S_2\}\cup \D^\d$$
is a $2$-starter in $\Z_{7n}$ relative to $7\Z_{7n}$.
\end{proof}

\begin{lem}\label{n12}
For any $m\equiv 3\pmod 4$, there exists a cyclic
$C_4$-factorization of $K_{m\times 12}$.
\end{lem}

\begin{proof}
Suppose firstly  $m\equiv 3\pmod 8$.
If $m=3$ the statement follows from Lemma \ref{lem:37}. So, we may assume
$m\geq 11$. Define $S_0=\s A_0$ and $S_2=\s A_2$, where
\begin{eqnarray*}
A_0 &=& \X_{0}^e\left(2m,\left\{\frac{3m-9}{8}\right\}\right)\cup
\X_{0}^o\left(2m,\left\{\frac{3m-9}{8}\right\}\right)\cup \\
&&\left\{\left[\frac{m-3}{4},\frac{m+1}{4}\right],
\left[\frac{7m-13}{4},\frac{7m-1}{4}\right],[m,2m-2]\right\},\\
A_2 &=& \Y_{2}(m+1,\emptyset)\cup \left(m+\Y_{2}(m-2,\emptyset)\right)
\cup \left\{\left[\frac{m-1}{2},2m-1\right]\right\}.
\end{eqnarray*}
Also, define
$S_4=\left(0,\frac{3m-5}{2}, 1,\frac{9m+7}{2} \right)\cup [3,10]_{6m+3}$.
Observe that $\phi(S_0)$ and  $\phi(S_2)$ are  both  transversals of the subgroup of index $2m$ in $\Z_{12m}$ and 
that $\phi(S_4)$ is a transversal of the subgroup of index $6$ in $\Z_{12m}$.
Furthermore, $\Delta A_0=\pm\left(\left(2+2\ii{m-2}\right)\setminus \right.$ $\left.\left\{\frac{3m-5}{2},\frac{3m-1}{2}\right\}\right)\cup \pm\{1,3,m-2\}$,
$\Delta A_2=\pm(2m+1+\ii{m-2})\cup \pm \{\frac{3m-1}{2}\}$ and
$\partial S_4=\pm\{\frac{3m-5}{2}, \frac{3m-7}{2}, \frac{9m+5}{2}, \frac{9m+7}{2}, 7, 6m-7 \}$.
Let $\D$ be the set of the odd integers  $d\in \ii{3m-1}\setminus \left(\{1,3,7,m-2, \frac{3m-7}{2}\}
\cup \left(2m+1+2\ii{\frac{m-3}{2}}\right)\cup m\Z
\right)$.

Suppose now $m\equiv 7\pmod 8$.
If $m=7$ the statement follows from Lemma \ref{lem:37}. So, we may assume
$m\geq 15$. Define $S_0=\s A_0$ and $S_2=\s A_2$, where
\begin{eqnarray*}
A_0 &=& \X_{0}^e\left(2m,\left\{\frac{3m-5}{8}\right\}\right)\cup
\X_{0}^o\left(2m,\left\{\frac{3m+3}{8}\right\}\right)\cup \\
&&\left\{\left[\frac{m-11}{4},\frac{m-7}{4}\right],
\left[\frac{7m-9}{4},\frac{7m+11}{4}\right],[m,2m-2]\right\},\\
A_2 &=& \Y_{2}(m+1,\emptyset)\cup \left(m+\Y_{2}(m-2,\emptyset)\right)
\cup \left\{\left[\frac{m-1}{2},2m-1\right]\right\}.
\end{eqnarray*}
Also, define
$S_4=\left(0,\frac{3m+11}{2}, -1,\frac{9m-13}{2} \right)\cup [2,9]_{6m+2}$.
Observe that $\phi(S_0)$ and  $\phi(S_2)$ are  both  transversals of the subgroup of index $2m$ in $\Z_{12m}$ and $\phi(S_4)$ is a transversal of the subgroup of index $6$ in $\Z_{12m}$.
Furthermore, $\Delta A_0=\pm( (2+2\ii{m-2})$
$\setminus \left\{\frac{3m-1}{2},\frac{3m+11}{2}\right\})\cup \pm\{1,5,m-2\}$,
$\Delta A_2=\pm(2m+1+\ii{m-2})\cup \pm \{\frac{3m-1}{2}\}$
and
$\partial S_4=\pm\{\frac{3m+11}{2}, \frac{3m+13}{2}, \frac{9m-11}{2}, \frac{9m-13}{2},$ $7, 6m-7 \}$.
Let $\D$ be the set of the odd integers  $d\in \ii{3m-1}\setminus\left(\{1,5,7, m-2, \frac{3m+13}{2}\}
\cup \left(2m+1+2\ii{\frac{m-3}{2}}\right)\cup m\Z
\right)$.

In both cases, $\Sigma=\{S_0,S_2,S_4\}\cup \D^\d$ is a $2$-starter in $\Z_{12m}$ relative to $m\Z_{12m}$.
\end{proof}

\begin{prop}\label{prop:m3n4}
For any $m\equiv 3\pmod 4$  and any $n\equiv 4 \pmod 8$ with $n>4$, there exists a cyclic
$C_4$-factorization of  $\Kmn$.
\end{prop}

\begin{proof}
We split the proof into two subcases according to the congruence class of $m$ modulo $8$.
By Lemmas \ref{lem:37} and \ref{n12} we may assume $m\geq 11$ and $n\geq 20$.
Let $n=8t+20$.
\smallskip

\noindent \underline{Case 1:} $m\equiv 3 \pmod{8}$.
Firstly, suppose $m=11$.
If $n=20$, define $S_0(k)=\s A_k$, for $k=0,1,2$, where
\begin{eqnarray*}
A_0 & =& \{[0,18],[4,14],[6,12],[8,10]\}\cup \{[1,21],[5,17],[7,15],[9,13]\}\cup\\
&&\{[11,20],[2,25],[16,19]\},\\
A_1 & =& \{[0,40],[2,38],[4,36],[6,34],[8,32]\}\cup\{[1,43],[3,41],[5,39],[7,37],\\
&&[9,35]\}\cup\{[11,42]\},\\
A_2 & =& \{[0,54],[1,53],[2,52],[3,51],[4,50]\}\cup \{[11,64],[12,63],[13,62],[14,61],\\
&&[15,60]\}\cup \{[5,21]\}.
\end{eqnarray*}
Also, define
$$S_2=(0,14,1,97)\cup [2,9]_{112}\cup[3,8]_{113}\cup[5,6]_{115}.$$
Notice that $\phi(S_0(k))$ and $\phi(S_2)$ are a transversal of the subgroup of index $22$ and of index $10$, respectively, in $\Z_{220}$.
It is easy to see that the set $\Sigma=\{S_0(0),S_0(1),S_0(2),$ $S_2\}\cup \D^\d$, where
$\D$ is the set of the odd integers of $\ii{54}\setminus \{11,33\}$ not appearing in $\partial S_i$, is a $2$-starter in $\Z_{220}$ relative to $11\Z_{220}$.

So, suppose $n>20$. For $k\in \ii{t+2}$ define $S_0(k)=\s A_k$ where
\begin{eqnarray*}
A_0 & =& \X_{0}^e\left(22,\{1\}\right)\cup
\X_{0}^o\left(22,\{3\}\right)\cup \left\{[3,28],[12,41],[11,20]\right\},\\
A_k & =& \X_{2k}^e(22,\emptyset)\cup \X_{2k}^o(22,\emptyset)\cup \{[11,20+22k]\}\quad
\textrm{ for } 1\leq k \leq t+1,\\
A_{t+2} &= &\Y_{2t+4}(12,\emptyset)\cup \left(11+\Y_{2t+4}(9,\emptyset)\right)
\cup \left\{[5,21]\right\}.
\end{eqnarray*}
Now let $S_2=\s B\cup (0,6, -1,  \frac{11n-14}{2})$,
where
$$B= \Y_0\left(21,\left\{3,4,5,10\right\}\right)\cup \left\{[14,71], [16,51]\right\}.$$
Notice that $\phi(S_0(k))$ and $\phi(S_2)$ are  both transversals of the subgroup of index $22$  in $\Z_{11n}$.
One can check that the set $\Sigma=\{S_0(k): k \in \ii{t+2}\}\cup \{S_2\}\cup \D^\d$, where
$\D$ is the set of the odd integers of $\ii{\frac{11n}{4}-1}\setminus 11\Z$ not appearing in $\partial S_i$, is a $2$-starter in $\Z_{11n}$ relative to $11\Z_{11n}$.

Assume now $m\geq 19$. For $k\in \ii{t+2}$ define $S_0(k)=\s A_k$ where
\begin{eqnarray*}
A_0& =& \X_{0}^e\left(2m,\left\{\frac{m-3}{8}\right\}\right)\cup
\X_{0}^o\left(2m,\left\{\frac{3m-9}{8}\right\}\right)\cup \\
&&\left\{\left[\frac{m+1}{4},
\frac{11m-9}{4}\right],\left[\frac{5m-7}{4},\frac{15m-1}{4}\right],[m,2m-2]\right\},\\
A_k&=&\X_{2k}^e(2m,\emptyset)\cup \X_{2k}^o(2m,\emptyset)\cup \{[m,2m-2+2km]\}\quad
\textrm{ for } 1\leq k \leq t+1,\\
A_{t+2} & =& \Y_{2t+4}(m+1,\emptyset)\cup \left(m+\Y_{2t+4}(m-2,\emptyset)\right)
\cup \left\{\left[\frac{m-1}{2},2m-1\right]\right\}.
\end{eqnarray*}
Also, let $S_2=\s B\cup(0,\frac{m+1}{2}, -1,  \frac{mn-m-3}{2})$,
where
\begin{eqnarray*}
B& =& \Y_0\left(2m-1,\left\{\frac{m+1}{4},\frac{m-3}{2},\frac{m-1}{2},m-1\right\}\right)\cup\\
&&\left\{\left[\frac{m-1}{2},\frac{13m+1}{4}\right], \left[ \frac{3m-1}{2},\frac{19m-5}{4}\right]\right\}.
\end{eqnarray*}
Notice that $\phi(S_0(k))$ and $\phi(S_2)$ are both transversals of the subgroup of index $2m$  in $\Z_{mn}$.
Furthermore, $\Delta A_0=\pm \left( (2+2\ii{m-2})\setminus \{\frac{m+1}{2}, \frac{3m-1}{2}\}\right)\cup  
\pm \{\frac{5m-5}{2},\frac{5m+3}{2},$ $m-2 \}$, $\Delta A_k=\pm (2+2km+2\ii{m-2})\cup \pm \{m-2+2km \}$ when $1\leq k\leq t+1$ and $\Delta A_{t+2}=\pm(1+(2t+4)m+\ii{m-2})\cup \pm \{\frac{3m-1}{2} \}$.
Finally, denoting $\mathcal{B}=\left(
(1+2\ii{m-2})\setminus\{ \frac{m+3}{2},m-2,m\}\right) \cup\{ \frac{11m+3}{4},\frac{13m-3}{4}\}$, we have $\Delta B=\pm \mathcal{B}$ and
$\partial S_2=\pm \{b,\frac{mn}{2}-b :b \in\mathcal{B} \}\cup \pm \{ \frac{m+1}{2},\frac{m+3}{2},\frac{mn-m-3}{2},\frac{mn-m-1}{2}\}$.
Since $m\geq 19$, $(\cup_k \partial S_0(k))\cup \partial S_2$ is a set.

Define $\Sigma=\{S_0(k): k \in \ii{t+2}\}\cup \{S_2\}\cup \D^\d$, where
$\D$ is the set of the odd integers of $\ii{\frac{mn}{4}-1}\setminus m\Z$ not appearing in $\partial S_i$.
It results that $\Sigma$ is a $2$-starter in $\Z_{mn}$ relative to $m\Z_{mn}$.
\smallskip

\noindent \underline{Case 2:} $m\equiv 7 \pmod{8}$.
For $k\in \ii{t+2}$ define $S_0(k)=\s A_k$ where
\begin{eqnarray*}
A_0& =& \X_{0}^e\left(2m,\left\{\frac{3m-5}{8}\right\}\right)\cup
\X_{0}^o\left(2m,\left\{\frac{m-7}{8}\right\}\right)\cup \\
&&\left\{\left[\frac{m-7}{4},
\frac{11m-1}{4}\right],\left[\frac{5m+1}{4},\frac{15m-9}{4}\right],[m,2m-2]\right\},\\
A_k & =& \X_{2k}^e(2m,\emptyset)\cup \X_{2k}^o(2m,\emptyset)\cup \{[m,2m-2+2km]\}\quad \textrm{
if } 1\leq k \leq t+1,\\
A_{t+2} &= &\Y_{2t+4}(m+1,\emptyset)\cup \left(m+\Y_{2t+4}(m-2,\emptyset)\right)
\cup \left\{\left[\frac{m-1}{2},2m-1\right]\right\}.
\end{eqnarray*}
Also, let $S_2=\s B\cup (0,\frac{m+1}{2}, -1,  \frac{mn-m-3}{2})$,
where
\begin{eqnarray*}
B& =& \Y_0\left(2m-1,\left\{\frac{m+1}{4},\frac{m-3}{2},\frac{m-1}{2},m-1\right\}\right)\cup\\
&&\left\{\left[\frac{m-1}{2},\frac{11m-5}{4}\right], \left[\frac{3m-1}{2}, \frac{21m+1}{4}\right]\right\}.
\end{eqnarray*}
Notice that $\phi(S_0(k))$ and $\phi(S_2)$ are both transversals of the subgroup of index $2m$  in $\Z_{mn}$.
Furthermore, $\Delta A_0=\pm \left( (2+2\ii{m-2})\setminus \{\frac{m+1}{2},\frac{3m-1}{2}\}\right)\cup  \pm \{\frac{5m-5}{2}, \frac{5m+3}{2},$
 $m-2\}$, $\Delta A_k=\pm (2+2km+2\ii{m-2})\cup \pm \{m-2+2km \}$ when $1\leq k\leq t+1$ and $\Delta A_{t+2}=\pm(1+(2t+4)m+\ii{m-2})\cup \pm \{\frac{3m-1}{2} \}$.
Finally, denoting $\mathcal{B}=\left(
(1+2\ii{m-2})\setminus\{ \frac{m+3}{2},m-2,m\}\right) \cup\{ \frac{9m-3}{4},\frac{15m+3}{4}\}$, we have $\Delta B=\pm \mathcal{B}$ and
$\partial S_2=\pm \{b,\frac{mn}{2}-b :b \in\mathcal{B} \}\cup \pm \{ \frac{m+1}{2},\frac{m+3}{2},\frac{mn-m-3}{2},\frac{mn-m-1}{2}\}$.
Since $m\geq 15$, $(\cup_k \partial S_0(k))\cup \partial S_2$ is a set.

The set $\Sigma=\{S_0(k): k \in \ii{t+2}\}\cup \{S_2\}\cup \D^\d$, where
$\D$ is the set of the odd integers of $\ii{\frac{mn}{4}-1}\setminus m\Z$ not appearing in $\partial S_i$, is a $2$-starter in $\Z_{mn}$ relative to $m\Z_{mn}$.
\end{proof}

\begin{ex}
Let $m=15$ and $n=20$, hence we are in Case 2 of the proof of Proposition \ref{prop:m3n4}. Following such a proof we
construct the following sets of edges
\begin{eqnarray*}
A_0 &=& \{[12,14],[10,16],[8,18],[6,20],[4,22],[0,26]\}\cup\\
& & \{[13,17],[9,21],[7,23],[5,25],[3,27],[1,29]\}\cup\{[2,41],[19,54],[15,28]\},\\
A_1 &=& \{[12,44],[10,46],[8,48],[6,50],[4,52],[2,54],[0,56]\}\cup\\
& &\{[13,47],[11,49],[9,51],[7,53],[5,55],[3,57],[1,59]\}\cup\{[15,58]\},\\
A_2 &=& \{[6,68],[5,69],[4,70],[3,71],[2,72],[1,73],[0,74]\}\cup\\
& &\{[21,82],[20,83],[19,84],[18,85],[17,86],[16,87],[15,88]\}\cup\{[7,29]\},\\
B & =& \{[14,15],[13,16],[12,17],[11,18],[9,20],[6,23],[5,24],[4,25],[3,26],\\
& &[2,27],[1,28]\} \cup\{[7,40],[22,79]\}.
\end{eqnarray*}
Hence it results $\D=\{29,31,37,41\}\cup(47+2\ii{4})\cup\{59\}$.
\end{ex}

\section{Existence of a cyclic hamiltonian $2$-factorization of $\Kmn$}\label{sec7}

In this section we consider the existence of a cyclic hamiltonian $2$-factorization of $\Kmn$ with an even number of vertices.
By Remark \ref{rem:neven}, $n$ is even and
in view of Theorem \ref{thm:MPP}
we can suppose $m$ odd.\\
Firstly we deal with the case $n\equiv 2 \pmod 4$. By Corollary \ref{cor:ne}, this implies that
$m \equiv 1 \pmod 4$.
As observed in the Introduction, $K_{m\times 2}=K_{2m}-I$ is  the cocktail party
graph;  thus we can suppose throughout this section that $n> 2$, since for $n=2$ we can rely on  Theorem
\ref{JM}.

The following result will be very useful for our constructions.

\begin{lem}\label{lem:248}
Let $n\equiv2 \pmod4$.
For any odd integer $x$ not divisible by $m$ and any integer $a\geq1$,
$[0,x]_{2^a}$ is a hamiltonian cycle of $\Kmn$.
\end{lem}

\begin{proof}
The thesis immediately follows by \cite[Remark 2.3]{MPP} since the order of
$2^a$ in $\Z_{mn}$ is $\frac{mn}{2}$ by the hypothesis on $m$ and $n$.
\end{proof}

\begin{rem}\label{rem:diff248}
Let $n\equiv2 \pmod 4$ and let $U=\{x_0,x_0+2^{a_0},x_1,x_1+2^{a_1},\ldots,x_b,x_b+2^{a_b}\}$ be a set  containing odd
integers
not divisible by $m$. Then  $W_u=[0,x_u+2^{a_u}]_{2^{a_u}}$ is a hamiltonian cycle of $\Kmn$ for any $u\in \ii{b}$.
Obviously, $\partial\left(\ccup_{u=0}^b W_u\right)=\pm U$.
\end{rem}

The following definition and lemma are instrumental in proving Theorem \ref{thm:n2}, where
we shall
settle the case $n\equiv 2\pmod 4$.

\begin{defini}\label{defi:esse}
For all positive integers $s,d$ and all odd integers $w\geq 3$, set
$$\R(s,d,w)=s+d\ii{\frac{w-3}{2}}$$
and
$$\varphi(s,d,w)=\left|\{x \in \R(s,d,w) : \gcd(x,w)=1\}\right|.$$
\end{defini}

\begin{lem}\cite[Lemma 4.6]{MPP} \label{lem:ok}
Assume  $\gcd(s,d,w)=1$. If $3\nmid s$ when $w=3$, then $\varphi(s,d,w)>0$.
\end{lem}

Now, we can prove the following.

\begin{prop}\label{thm:n2}
Let $m$ be an odd integer; let $n \equiv 2 \pmod 4$ and $n>2$. Then there exists a cyclic hamiltonian $2$-factorization of $\Kmn$
 if and only if $m\equiv 1\pmod 4$.
\end{prop}

\begin{proof}
The non existence immediately follows from Corollary \ref{cor:ne}.\\
So assume $m\equiv 1\pmod 4$. In order to prove the existence we will construct a
$2$-starter $\Sigma$ in $\Z_{mn}$ relative to $m\Z_{mn}$, where each $S\in \Sigma$ is a hamiltonian cycle.
We have to distinguish two cases according to the congruence class of $n$ modulo $8$.
\smallskip

\noindent \underline{Case 1:} $n=8t+10$.\\
For any $k \in \ii{t}$
let
\begin{eqnarray}\label{cycleAi}
A_k & =  &[0,4m-2 +4km,2, 4m-4 +4km,4, \ldots,4m-(m-1)+4km,\\
\nonumber &&m-1, 4m-m+4km, m+2, 4m-(m+2)+4km,m+4,\\
\nonumber &&4m-(m+4)+4km,\ldots,  2m-1,4m-(2m-1)+4km]_{2m}
\end{eqnarray}
and
\begin{equation}\label{cycleB}
B=\left[0,\frac{mn}{2}-1,1,\frac{mn}{2}-2,2,\ldots,\frac{mn}{2}-\frac{m-1}{2},\frac{m-1}{2}\right]
_m.
\end{equation}
Note that $\phi(A_k)$ and $\phi(B)$ are transversals of the subgroup of index
$2m$ and of index $m$, respectively, of $\Z_{mn}$.
Also, we have
$$
\partial \left(\ccup_{k=0}^{t} A_k \cup B\right)  = \pm
\left(\left(\left(2+2\ii{(2t+2)m-2}\right) \setminus \left(2m+2m\ii{2t}\right)\right)\cup\right.$$
$$\qquad \left(1+2m\ii{2t+1} \right)
\cup\left\{\frac{m+1}{2}\right\}\cup\left.\left((4t+4)m+1+\ii{m-2}\right) \right).
$$
It is easy to see that the number of odd integers of the set
$\Theta=\ii{\frac{mn}{2}}\setminus (m\ii{\frac{n}{2}}\cup\partial (\cup_{k}A_k \cup B))$ is odd.

\noindent \underline{Case 1.a:} Assume that $3$ does not divide $m$.
Consider the set $\R(2m-3,2m,\frac{n}{2})$ whose elements are all coprime with $m$. Since $\gcd(2m-3,2m,\frac{n}{2})=1$
and $\gcd(3,$ $2m-3)=1$ we can apply Lemma \ref{lem:ok}.
Hence there exists an element $\kappa$ in $\R(2m-3,2m,\frac{n}{2})$ coprime with $\frac{n}{2}$.
So $\kappa$ is coprime with $mn$  and $C=[0]_\kappa$ is a hamiltonian cycle of $\Kmn$.

If $m\geq 13$, then the set $\Theta\setminus\{\kappa\}$  satisfies
the condition of Remark \ref{rem:diff248}, so we construct the associated cycles $W_u$.

The case $m=5$ (and hence $2m-3=7$) requires a special analysis. If $n=10$ take the cycles
$D_0=[0,9]_2$, $D_1=[0,19]_2$ and, instead of $C=[0]_\kappa$, take $C=[0]_{13}$.
It is easy to see that $\Sigma=\{A_0,B,C,D_0,D_1\}$ is a $2$-starter in $\Z_{50}$
relative to $5\Z_{50}$.
So we can suppose $n\geq 18$. If $\kappa=20j+7$ with $j\in\ii{t}$
the elements of $\Theta\setminus\{\kappa\}$ can be partitioned into $2$ sets $\Theta_+$ and $\Theta_-$ of even size
containing respectively the elements of $\Theta$ greater than $\kappa$
and  the elements of $\Theta$ less than $\kappa$. Note that if $\kappa=7$ then $\Theta_-$ is empty.
The sets $\Theta_+$ and $\Theta_-$  satisfy the hypothesis of Remark \ref{rem:diff248}, and
hence we construct the cycles $W_u$ as in the remark.

If $\kappa=10(2j+1)+7$ with $j\in\ii{t-1}$ take
 $D_0=[0,\kappa+10]_{8}$, $D_1=[0,\kappa+12]_{16}$, $D_2=[0,\kappa+20]_{4}$
 and  $D_3=[0,\kappa+22]_{16}$, which are hamiltonian cycles by Lemma \ref{lem:248}.
The elements of $\Theta\setminus\partial\left(C\cup (\cup_i D_i)\right)$ can be partitioned into $2$ sets $\Theta_+$ and
$\Theta_-$ of even size
containing respectively the elements of $\Theta$ greater than $\kappa$
and  the elements of $\Theta$ less than $\kappa$.
Reasoning as above, we can construct the cycles $W_u$ as in Remark \ref{rem:diff248}.

Finally, if $\kappa=10(2j+1)+7$ with $j=t$ take
 $D_0=[0,\kappa-20]_{4}$, $D_1=[0,\kappa-10]_{8}$, $D_2=[0,\kappa-4]_{4}$
 and  $D_3=[0,\kappa+2]_{16}$, which are hamiltonian cycles by Lemma \ref{lem:248}.
The set $\Theta\setminus\partial\left(C\cup (\cup_i D_i) \right)$ is of even size
and satisfies the conditions of Remark \ref{rem:diff248}, so we construct the cycles $W_u$.

\noindent \underline{Case 1.b:} Assume that $3$ divides $m$ (which implies $m\geq 9$). Consider the set
$\R(2m-1,2m,\frac{n}{2})$ whose elements are all coprime with $m$.
Since $\gcd(2m-1,2m,\frac{n}{2})=1$ and $\gcd(3,2m-1)=1$ we can apply Lemma \ref{lem:ok}.
Hence there exists an element $\kappa$ in $\R(2m-1,2m,\frac{n}{2})$ coprime with $\frac{n}{2}$.
So $\kappa$ is coprime with $mn$ and  $C=[0]_\kappa$ is a hamiltonian cycle of $\Kmn$.

If $\kappa=4m-1+4jm$ with $j\in\ii{t}$ the elements of $\Theta\setminus\{\kappa\}$ can be partitioned, as above,
into 2 sets $\Theta_+$ and $\Theta_-$ of even size.
Hence we construct the cycles $W_u$ as in Remark \ref{rem:diff248}.

If $\kappa=2m-1+4jm$ with $j\in\ii{t}$ take $D_0=[0,\kappa+6]_8$.
Note that the set $\Theta\setminus\partial \left(C \cup D \right)$ can be divided into sets of even size
which satisfy the conditions of Remark \ref{rem:diff248}. So, again, we construct the cycles $W_u$.
\smallskip

In both cases 1.a and 1.b, $\Sigma=\{ A_i, B, C,  D_v, W_u\}$
is a $2$-starter in $\Z_{mn}$ relative to $m\Z_{mn}$.
\smallskip

\noindent \underline{Case 2:} $n=8t+6$.\\
Take
\begin{equation}\label{cycleA}
A=[0,2m-2,2,2m-4,4,\ldots,m+1,m-1]_m,
\end{equation}
$B$ as in \eqref{cycleB} and for any $k\in \ii{t-1}$
take

\begin{eqnarray}\label{cycleC}
C_k&= & [0,6m-2+4km,2,6m-4+4km,4,\ldots,5m+1+4km,\\
\nonumber &&m-1,5m+4km,m+2,5m-2+4km,m+4,\ldots,\\
\nonumber &&2m-1,4m+1+4km]_{2m}.
\end{eqnarray}
Note that $$\Lambda=\partial (A \cup B \cup
\cup_k C_k)= \pm\left(\left(\left(2+2\ii{m(2t+1)-2}\right)\setminus\left(2m+2m\ii{2t-1}\right)\right) \right. \cup$$
$$\left. \left(1+2m\ii{2t}\right)\cup\left\{\frac{m+1}{2}\right\}\cup\left(1+2m(2t+1)+\ii{m-2}\right) \right).$$
One can check that $\ii{\frac{mn}{2}}\setminus(m\Z \cup\Lambda)$ is a set satisfying the hypothesis of Remark \ref{rem:diff248}.
So we construct the cycles $W_u$.

It is not hard to see that $\Sigma=\{A,  B, C_k, W_u\}$
is a $2$-starter in $\Z_{mn}$ relative to $m\Z_{mn}$.
\end{proof}

\begin{ex}
Following the proof of Theorem \ref{thm:n2}, we construct a $2$-starter $\Sigma$ in $\Z_{mn}$ relative to $m\Z_{mn}$
for some choice of $m$ and $n$.\\
\noindent \underline{Case 1.a:} Let $m=5$ and $n=18$. We obtain
$$A_0=[0,18,2,16,4,15,7,13,9,11]_{10},\quad A_1=[0,38,2,36,4,35,7,33,9,31]_{10},$$
$$B=[0,44,1,43,2]_5.$$
It results
$\partial(A_0\cup A_1\cup B)=\pm\left(\left((2+2\ii{18})\setminus\{10,20,30\}\right)\cup\{1,11,21,31\}\cup\{3\}\right.$ $\left.\cup\{41,42,43,44\}\right)$
hence $\Theta=\{7,9,13,17,19,23,27,29,33,37,39\}$. Note that $\Theta$ has odd size.
Since $3$ does not divide $m$ we consider $\R(7,10,9)=\{7,17,27,37\}$.
We can choose for instance
$\kappa=17$ which is coprime with $\frac{n}{2}=9$, so we take $C=[0]_{17}$.
Now we have to take the cycles $D_0=[0,27]_8$, $D_1=[0,29]_{16}$, $D_2=[0,37]_4$
and $D_3=[0,39]_{16}$. It results $\Theta\setminus\partial\left(C \cup (\cup_i D_i)\right)=\{7,9\}$.
So in this case $\Theta_+=\emptyset$. It remains to consider only the cycle $W_0=[0,9]_2$.
Hence take $\Sigma=\{ A_i, B, C, D_v,W_0\}$.\\
\noindent \underline{Case 1.b:} Let $m=9$ and $n=26$. We obtain
$$A_0=[0,34,2,32,4,30,6,28,8,27,11,25,13,23,15,21,17,19]_{18},$$
$$A_1=[0,70,2,68,4,66,6,64,8,63,11,61,13,59,15,57,17,55]_{18},$$
$$A_2=[0,106,2,104,4,102,6,100,8,99,11,97,13,95,15,93,17,91]_{18},$$
$$B=[0,116,1,115,2,114,3,113,4]_9.$$
We have $\partial(\cup_i A_i\cup
B)=\pm(((2+2\ii{52})\setminus\{18,36,54,72,90\})\cup\{1,19,37,55,73,91\}$
$\cup\{5\}\cup(109+\ii{7}))$. Hence $\Theta=
(1+2\ii{53})\setminus\{1,5,9,19,27,37, 45, 55, 63,73, 81, 91,$ $99\}$.
 Since $3$ divides $m$ we consider $\R(17,18,13)=\{17,35,53,71,89,107\}$.
All the elements of $\R(17,18,13)$ are coprime with $\frac{n}{2}$.
The more convenient choices are $35$, $71$ or $107$ since in each of these cases we have not to construct
the cycle $D_0$. Choosing $\kappa=107$, we have to take $C=[0]_{107}$
and following Remark \ref{rem:diff248}, we construct
$W_0=[0,7]_4$, $W_1=[0,13]_2$, $W_2=[0,17]_2$, $W_3=[0,23]_2$, $W_4=[0,29]_4$, $W_5=[0,33]_2$, $W_6=[0,39]_4$,
$W_7=[0,43]_2$, $W_8=[0,49]_2$, $W_9=[0,53]_2$, $W_{10}=[0,59]_2$, $W_{11}=[0,65]_4$,
$W_{12}=[0,69]_2$, $W_{13}=[0,75]_4$, $W_{14}=[0,79]_2$, $W_{15}=[0,85]_2$, $W_{16}=[0,89]_2$,
$W_{17}=[0,95]_2$, $W_{18}=[0,101]_4$, $W_{19}=[0,105]_2$.
Hence take $\Sigma=\{A_i, B, C,W_u\}$.\\
\noindent \underline{Case 2:} Let $m=13$ and $n=14$. We obtain
$$A=[0,24,2,22,4,20,6,18,8,16,10,14,12]_{13},$$
$$B=[0,90,1,89,2,88,3,87,4,86,5,85,6]_{13},$$
\begin{eqnarray*}
C_0 & =& [0,76,2,74,4,72,6,70,8,68,10,66,12,65,15,63,17,61,19,59,21,57,\\
&&23,55,25,53]_{26}.
\end{eqnarray*}
It results
$\Lambda=\partial(A\cup B\cup C_0)=\pm(((2+2\ii{37})\setminus\{26,52\})
\cup\{1,7,27,53\}\cup(79+\ii{11}))$.
Hence we are left to consider $\ii{91}\setminus(13\Z \cup \Lambda)=
(3+2\ii{37})\setminus \{7, 13,27,39, 53, 65\}$.
Now, following Remark \ref{rem:diff248}, we construct the cycles
$W_0=[0,5]_2$, $W_1=[0,11]_2$, $W_2=[0,17]_2$, $W_3=[0,21]_2$, $W_4=[0,25]_2$,
$W_5=[0,31]_2$, $W_6=[0,35]_2$, $W_7=[0,41]_4$, $W_{8}=[0,45]_2$, $W_{9}=[0,49]_2$,
$W_{10}=[0,55]_4$, $W_{11}=[0,59]_2$, $W_{12}=[0,63]_2$, $W_{13}=[0,69]_2$, $W_{14}=[0,73]_2$, $W_{15}=[0,77]_2$.
Hence take $\Sigma=\{A, B,C_0, W_u\}$.
\end{ex}

Now we consider the case $n\equiv 0\pmod 4$.

\begin{prop}\label{thm:n4}
Let $m$ be an odd integer;
for any $n\equiv 0 \pmod 4$
there exists a cyclic hamiltonian $2$-factorization of $K_{m \times n}$.
\end{prop}

\begin{proof}
Let us first consider the case $n=4$.
Take $A$ as in \eqref{cycleA} so that $\partial A=\pm\left(\{1\}\cup(2+2\ii{m-2})\right)$.

If $m\equiv 1 \pmod 4$, take
$F_j=[0,4j+5]_2$ for $j\in \ii{\frac{m-3}{2}}\setminus\left\{\frac{m-5}{4}\right\}$ and $C=[0]_{m-2}$ (note that we have
$\gcd(m-2,4m)=1$).

If $m\equiv 3 \pmod 4$, take
$$F_j=\left\{\begin{array}{ll}
{[0,4j+5]_2} &  \textrm{ for } j\in \ii{\frac{m-7}{4}},\\[2pt]
{[0,4j+3]_2} &  \textrm{ for } j\in\frac{m+1}{4}+\ii{\frac{m-7}{4}},
\end{array}\right.
$$
and $C=[0]_{2m-1}$. Note  that $\gcd(2m-1,4m)=1$.
In both cases  $\Sigma=\{A,  C, F_j\}$
is a $2$-starter in $\Z_{4m}$ relative to $m\Z_{4m}$.

Suppose now $n\geq 8$. We split the proof into two similar cases according to the congruence
class of $m$ modulo $4$.
\smallskip

\noindent \underline{Case 1:} $m\equiv 1 \pmod 4$.\\
\indent If $n\equiv 4 \pmod 8$, take $A$ as in \eqref{cycleA}
and for $k\in \ii{\frac{n-12}{8}}$ take $C_k$ as in \eqref{cycleC}.
It is easy to see that
$$\partial (A \cup (\cup_k
C_k))=\pm\left(\left(2\ii{\frac{mn}{4}-1}\setminus 2m\ii{\frac{n-4}{4}}\right)\cup\left(1+2m\ii{\frac{n-4}{4}}\right)\right).$$
Now, for any $i\in\ii{\frac{n-4}{4}}$ and any $j\in\ii{\frac{m-3}{2}}\setminus \left\{\frac{m-5}{4}\right\}$ we take
$$D_{i,j}=[0,2mi+4j+5]_2.$$
The differences arising from the cycles $D_{i,j}$ are the odd integers
$$\pm\left(\left(1+2\ii{\frac{mn}{4}-1}\right)\setminus \left\{1+2mi, m-2+2mi,m+2mi: i \in \ii{\frac{n-4}{4}}\right\}\right).$$
The only differences left to consider are the odd integers in the set
$\pm\left(m-2+2m\right.$ $\left.\ii{\frac{n-4}{4}}\right)$.
To obtain them, take $G=[0]_{mn/4-2}$ and finally  for $s\in \ii{\frac{n-12}{8}}$ take
$$F_s=[0,(2s+1)m-2]_{mn/2-4};$$
the $F_s$'s and $G$ are hamiltonian cycles, since $\gcd\left(\frac{mn}{4}-2,mn\right)=1$.
One can check that set
$\Sigma=\{A,  C_k, D_{i,j}, F_s,G\}$
is a $2$-starter in $\Z_{mn}$ relative to $m\Z_{mn}$.
\smallskip

If $n\equiv 0\pmod 8$, for $k\in\ii{\frac{n-8}{8}}$
take the cycles
\begin{eqnarray}
\widetilde{C}_k & =&
[0,4m-2+4km,2,4m-4+4km,4,\ldots,3m+1+4km,\\
\nonumber && m-1,4m+1+4km, 2m+3,4m-1+4km,2m+5,\ldots,3m,\\
&&3m+2+4km]_{2m}.\nonumber
\end{eqnarray}
Also, for any $i\in\ii{\frac{n-4}{4}}$ and for any $j\in \ii{\frac{m-3}{2}}\setminus \left\{
\frac{m-1}{4}\right\}$
let $D_{i,j}=[0,2mi+4j+3]_2$.
Finally, take
$$G_0=[0]_{mn/4-1},\quad G_1=[0]_{mn/2-1},$$
and for all $s\in \ii{\frac{n-16}{8}}$ consider
$$F_s=[0,2(s+1)m-1]_{mn/2-2}.$$
It is not hard to see that
$\Sigma=\{\widetilde{C}_k, D_{i,j}, F_s,G_0,G_1\}$
is a $2$-starter in $\Z_{mn}$ relative to $m\Z_{mn}$.
\smallskip

\noindent \underline{Case 2:} $m\equiv 3 \pmod 4$.\\
\indent If $n\equiv 4\pmod 8$, take the cycle $A$ as in \eqref{cycleA} and for
any $k\in \ii{\frac{n-12}{8}}$ take $C_k$ as in \eqref{cycleC}.
Also, for any $i\in \ii{\frac{n-4}{4}}$ and for any $j\in \ii{\frac{m-3}{2}}\setminus\{
\frac{m-3}{4}\}$ let  $D_{i,j}=[0,2mi+4j+5]_2$.
Finally, take $G=[0]_{mn/4+2}$ and  for $s\in \ii{\frac{n-12}{8}}$ take also the cycles
$$F_s=[0,(2s+1)m+2]_{mn/2+4}.$$
One can check that $\Sigma=\{A,C_k, D_{i,j}, F_s,G\}$
is a $2$-starter in $\Z_{mn}$ relative to $m\Z_{mn}$.
\smallskip

If $n\equiv 0\pmod 8$, for $k\in \ii{\frac{n-8}{8}}$  take the cycles $A_k$ as in \eqref{cycleAi}.
Also, for any $i\in \ii{\frac{n-4}{4}}$ take
$$D_{i,j}=\left\{\begin{array}{ll}
{[0,2mi+4j+5]_2} &  \textrm{ for } j\in \ii{\frac{m-7}{4}},\\[2pt]
{[0,2mi+4j+3]_2} &  \textrm{ for } j\in \frac{m+1}{4}+\ii{\frac{m-7}{4}}.
\end{array}\right.
$$
Finally, take
$$G_0=[0]_{mn/4-1},\quad G_1=[0]_{mn/2-1},$$
and for all $s\in\ii{\frac{n-16}{8}}$ take
$$F_s=[0,2(s+1)m-1]_{mn/2-2}.$$
It is not hard to see that
$\Sigma=\{A_k, D_{i,j}, F_s,G_0,G_1\}$
is a $2$-starter in $\Z_{mn}$ relative to $m\Z_{mn}$.
\end{proof}

\begin{ex}
Following the proof of Theorem \ref{thm:n4}, we construct a $2$-starter $\Sigma$ in
 $\Z_{mn}$ relative to $m\Z_{mn}$ for some choice of $m$ and $n$.\\
\noindent \underline{Case 1:} Let $m=9$ and $n=20$. The elements of $\Sigma$ are
$$A=[0,16,2,14,4,12,6,10,8]_9,$$
$$C_0=[0,52,2,50,4,48,6,46,8,45,11,43,13,41,15,39,17,37]_{18},$$
$$C_1=[0,88,2,86,4,84,6,82,8,81,11,79,13,77,15,75,17,73]_{18},$$
$$D_{0,0}=[0,5]_2, D_{0,2}=[0,13]_2, D_{0,3}=[0,17]_2,
D_{1,0}=[0,23]_2,  D_{1,2}=[0,31]_2,$$ $$D_{1,3}=[0,35]_2,
D_{2,0}=[0,41]_2, D_{2,2}=[0,49]_2, D_{2,3}=[0,53]_2,
D_{3,0}=[0,59]_2,$$ $$D_{3,2}=[0,67]_2, D_{3,3}=[0,71]_2,
D_{4,0}=[0,77]_2, D_{4,2}=[0,85]_2,D_{4,3}=[0,89]_2,$$
$$G=[0]_{43},\quad F_0=[0,7]_{86},\quad F_1=[0,25]_{86}.$$
\noindent \underline{Case 2:} Let $m=15$ and $n=12$. The elements of $\Sigma$ are

$$A=[0,28,2,26,4,24,6,22,8,20,10,18,12,16,14]_{15},$$
\begin{eqnarray*}
C_0&=&[0,88,2,86,4,84,6,82,8,80,10,78,12,76,14,75,17,73,19,71,21,69,23,\\
&&67, 25,65,27,63,29,61]_{30},
\end{eqnarray*}
$$D_{0,0}=[0,5]_2, D_{0,1}=[0,9]_2, D_{0,2}=[0,13]_2, D_{0,4}=[0,21]_2, D_{0,5}=[0,25]_2,$$
$$D_{0,6}=[0,29]_2,D_{1,0}=[0,35]_2, D_{1,1}=[0,39]_2, D_{1,2}=[0,43]_2,D_{1,4}=[0,51]_2,$$
$$D_{1,5}=[0,55]_2, D_{1,6}=[0,59]_2,D_{2,0}=[0,65]_2, D_{2,1}=[0,69]_2, D_{2,2}=[0,73]_2,$$
$$D_{2,4}=[0,81]_2, D_{2,5}=[0,85]_2, D_{2,6}=[0,89]_2,G=[0]_{47}, F_0=[0,17]_{94}.$$
\end{ex}

\section{Proof of Theorems \ref{thm:C4} and \ref{thm:Cmn}}\label{sec8}
\begin{proof}[Proof of Theorem \ref{thm:C4}]
``$\Rightarrow$'' It is Corollary \ref{cor:neC4}.\\
``$\Leftarrow$'' Case (a) is settled in Proposition \ref{prop:mnpariC4}.
Case (b) is considered in Propositions \ref{prop:modd}, \ref{prop:mx4}, 
\ref{prop:m1n4} and \ref{prop:m3n4}.
\end{proof}

\begin{proof}[Proof of Theorem \ref{thm:Cmn}]
``$\Rightarrow$'' It is sufficient to apply Theorem \ref{JM}, Remark \ref{rem:neven} and Corollary \ref{cor:ne}.\\
``$\Leftarrow$'' It follows from Theorem \ref{thm:MPP} and Propositions \ref{thm:n2} and \ref{thm:n4}.
\end{proof}

\section*{Acknowledgements}
The authors would like to thank Francesca Merola for the interesting discussions
and for her suggestions about the hamiltonian case.


\begin{thebibliography}{19}

\bibitem{BDD} D. Bryant, P. Danziger \and M. Dean,
On the Hamiltonian-Waterloo problem for bipartite $2$-factors,
J. Combin. Des. 21 (2013), 60--80.

\bibitem{BDP} D. Bryant, P. Danziger \and W. Pettersson,   Bipartite 2-factorizations of complete multipartite graphs,   J. Graph Theory 78 (2015), 287--294.

\bibitem{B} M. Buratti, Abelian $1$-factorizations of the complete graph,
European J. Combin. 22 (2001), 291--295.

 \bibitem{BDF2004} M. Buratti \and A. Del Fra,
 Cyclic hamiltonian cycle systems of the complete graph,
 Discrete Math. 279 (2004), 107--119.

\bibitem{BR} M. Buratti \and G. Rinaldi,
On sharply vertex transitive $2$-factorizations of the complete graph,
J. Combin. Theory A 111 (2005), 245--256.

\bibitem{CEZKVE} N. Cavenagh, S. El-Zanati, A. Khodkar \and C. Vanden Eynden,
On a generalization of the
Oberwolfach problem, J. Combin. Theory A 106 (2004), 255--275.
\newpage
\bibitem{EZTVE} S. El-Zanati, S.K. Tipnis \and C. Vanden Eynden,
A generalization of the Oberwolfach
problem, J. Graph Theory 41 (2002), 151--161.

\bibitem{G} P. Gvozdjak, On the Oberwolfach problem for complete multigraphs,
Discrete Math. 173 (1997), 61--69.

\bibitem{JM} H. Jordon \and J. Morris,
 Cyclic hamiltonian cycle systems of the complete graph minus $1$-factor,
 Discrete Math. 308 (2008), 2440--2449.

\bibitem{L2000} J. Liu,
 A generalization of the Oberwolfach problem and $C_t$-factorization of complete equipartite graphs,
 J. Combin. Des. 8 (2000), 42--49.

\bibitem{L2003} J. Liu,
 The equipartite Oberwolfach problem with uniform tables,
 J. Combin Theory A 101 (2003), 20--34.

\bibitem{LL} J. Liu \and D.R. Lick, On $\lambda$-fold equipartite Oberwolfach problem with uniform table sizes,
Ann. Comb. 7 (2003), 315--323.

\bibitem{M} G. Mazzuoccolo,
Primitive 2-factorizations of the complete graph,
Discrete Math. 308 (2008), 175--179.

\bibitem{MPP} F. Merola, A. Pasotti \and M.A. Pellegrini,
Cyclic and symmetric hamiltonian cycle systems of the complete multipartite graph: even number of parts, to appear on Ars Math. Contemp.

\bibitem{OP} M.A. Ollis \and D.A. Preece, Sectionable terraces and the (generalised) Oberwolfach problem,
Discrete Math. 266 (2003), 399--416.

\bibitem{PP} A. Pasotti \and M.A. Pellegrini,
Symmetric $1$-factorizations of the complete graph,
European J. Combin. 31 (2010), 1410--1418.

\bibitem{P} W. Piotrowski, The solution of the bipartite analogue of the Oberwolfach problem,
Discrete Math. 97 (1991), 339--356.

\bibitem{R} G. Rinaldi,
Nilpotent $1$-factorizations of the complete graph,
J. Combin. Des. 13 (2005), 393--405.

\end{thebibliography}
\end{document}